\documentclass[reqno,12pt,letterpaper]{amsart}
\usepackage[proof]{sdmacros}

\DeclareMathOperator{\Sp}{Sp}

\title[Dolgopyat's method and the fractal uncertainty principle]%
{Dolgopyat's method and\\
the fractal uncertainty principle}
\author{Semyon Dyatlov}
\email{dyatlov@math.mit.edu}
\address{Department of Mathematics, Massachusetts Institute of Technology,
77 Massachusetts Ave, Cambridge, MA 02139}
\author{Long Jin}
\email{long249@purdue.edu}
\address{Department of Mathematics, Purdue University,
150 N. University St, West Lafayette, IL 47907}

\begin{document}

\begin{abstract}
We show a fractal uncertainty principle with exponent ${1\over 2}-\delta+\varepsilon$,
$\varepsilon>0$,
for Ahflors--David regular subsets of $\mathbb R$ of dimension $\delta\in (0,1)$. This improves
over the volume bound ${1\over 2}-\delta$, and
$\varepsilon$ is estimated explicitly in terms of the regularity constant
of the set. The proof uses a version of techniques originating in the works of Dolgopyat, Naud, and Stoyanov
on spectral radii of transfer operators. Here the group invariance
of the set is replaced by its fractal structure. As an application, we quantify the result of Naud on
spectral gaps for convex co-compact hyperbolic surfaces and obtain a new spectral
gap for open quantum baker maps.
\end{abstract}

\maketitle

\addtocounter{section}{1}
\addcontentsline{toc}{section}{1. Introduction}


A \emph{fractal uncertainty principle} (FUP) states that no function can be localized close to
a fractal set in both position and frequency. Its most basic form is
\begin{equation}
  \label{e:fup-simple}
\|\indic_{\Lambda(h)}\mathcal F_h\indic_{\Lambda(h)}\|_{L^2(\mathbb R)\to L^2(\mathbb R)}=\mathcal O(h^\beta)\quad\text{as }h\to 0
\end{equation}
where $\Lambda(h)$ is the $h$-neighborhood of a bounded set $\Lambda\subset\mathbb R$,
$\beta$ is called the \emph{exponent} of the uncertainty principle,
and $\mathcal F_h$ is the semiclassical Fourier transform:
\begin{equation}
  \label{e:F-h}
\mathcal F_h u(\xi)=(2\pi h)^{-1/2}\int_{\mathbb R} e^{-ix\xi/h}u(x)\,dx.
\end{equation}
We additionally assume that $\Lambda$ is an Ahlfors--David regular set (see Definition~\ref{d:regular-set}) of
dimension $\delta\in (0,1)$ with some regularity constant $C_R>1$. Using the bounds $\|\mathcal F_h\|_{L^2\to L^2}=1$, $\|\mathcal F_h\|_{L^1\to L^\infty}\leq h^{-1/2}$,
the Lebesgue volume bound $\mu_L(\Lambda(h))\leq Ch^{1-\delta}$, and H\"older's inequality,
it is easy to obtain~\eqref{e:fup-simple} with $\beta=\max(0,{1\over 2}-\delta)$.

Fractal uncertainty principles were applied by Dyatlov--Zahl~\cite{hgap}, Dyatlov--Jin~\cite{oqm},
and Bourgain--Dyatlov~\cite{fullgap} to the problem
of essential spectral gap in quantum chaos: \emph{which open quantum chaotic systems have
exponential decay of local energy at high frequency?} 
A fractal uncertainty principle can be used to show local energy decay $\mathcal O(e^{-\beta t})$,
as was done for convex co-compact hyperbolic quotients in~\cite{hgap}
and for open quantum baker's maps in~\cite{oqm}. Here $\Lambda$ is related to the set of all trapped classical trajectories of the system and~\eqref{e:fup-simple} needs to be replaced by a more general statement,
in particular allowing for a different phase in~\eqref{e:F-h}.
The volume bound $\beta={1\over 2}-\delta$ corresponds to the Patterson--Sullivan gap or more generally,
the \emph{pressure gap}. See~\S\S\ref{s:hyper},\ref{s:oqm} below for a more detailed discussion.

A natural question is: can one obtain~\eqref{e:fup-simple} with $\beta>\max(0,{1\over 2}-\delta)$,
and if so, how does the size of the improvement depend on $\delta$ and $C_R$?
Partial answers to this question have been obtained in the papers mentioned above:
\begin{itemize}
\item \cite{hgap} obtained FUP with $\beta>0$ when $|\delta-{1\over 2}|$ is small
depending on $C_R$, and gave the bound $\beta>\exp(-\mathbf K (1+\log^{14}C_R))$
where $\mathbf K$ is a global constant;
\item \cite{fullgap} proved FUP with $\beta>0$ in the entire range $\delta\in (0,1)$,
with no explicit bounds on the dependence of $\beta$ on $\delta,C_R$;
\item \cite{oqm} showed that discrete Cantor sets satisfy FUP
with $\beta>\max(0,{1\over 2}-\delta)$
in the entire range $\delta\in (0,1)$ and obtained quantitative lower bounds on the
size of the improvement~-- see~\S\ref{s:oqm} below.
\end{itemize}
Our main result, Theorem~\ref{t:main}, shows that FUP holds
with $\beta>{1\over 2}-\delta$
in the case $\delta\in (0,1)$, and gives bounds
on $\beta-{1\over 2}+\delta$ which are polynomial in $C_R$ and thus stronger than the ones
in~\cite{hgap}. Applications include
\begin{itemize}
\item an essential spectral gap for convex co-compact hyperbolic surfaces
of size $\beta>{1\over 2}-\delta$, recovering and making quantitative
the result of Naud~\cite{NaudGap}, see~\S\ref{s:hyper};
\item an essential spectral gap of size $\beta>\max(0,{1\over 2}-\delta)$
for open quantum baker's maps, extending the result of~\cite{oqm} to
matrices whose sizes are not powers of the base, see~\S\ref{s:oqm}.
(For the case $\delta>{1\over 2}$ we use the results of~\cite{fullgap} rather than Theorem~\ref{t:main}.)
\end{itemize}

\subsection{Statement of the result}

We recall the following definition of Ahlfors--David regularity, which requires
that a set (or a measure) has the same dimension $\delta$ at all points and on a range of scales:
\begin{defi}
  \label{d:regular-set}
Let $X\subset \mathbb R$ be compact, $\mu_X$ be a finite measure
supported on $X$, and $\delta\in [0,1]$. We say that $(X,\mu_X)$ is \textbf{$\delta$-regular}
up to scale $h\in [0,1)$
with regularity constant $C_R\geq 1$ if
\begin{itemize}
\item for each interval $I$ of size $|I|\geq h$, we have
$\mu_X(I)\leq C_R |I|^\delta$;
\item if additionally $|I|\leq 1$ and the center of $I$ lies in~$X$, then $\mu_X(I)\geq C_R^{-1}|I|^\delta$.
\end{itemize}
\end{defi}
Our fractal uncertainty principle has a general form
which allows for two different sets $X,Y$ of different dimensions in~\eqref{e:fup-simple},
replaces the Lebesgue measure by the fractal measures $\mu_X,\mu_Y$,
and allows a general nondegenerate phase and amplitude in~\eqref{e:F-h}:
\begin{theo}
  \label{t:main}
Assume that
 $(X,\mu_X)$ is $\delta$-regular, and $(Y,\mu_Y)$ is $\delta'$-regular, up to scale $h\in(0,1)$ with constant $C_R$, where $0<\delta,\delta'<1$, and $X\subset I_0,Y\subset J_0$ for some intervals $I_0,J_0$.
Consider an operator
$\mathcal B_h:L^1(Y,\mu_Y)\to L^\infty(X,\mu_X)$ of the form
\begin{equation}
  \label{e:B-h}
\mathcal B_h f(x)=\int_Y \exp\Big({i\Phi(x,y)\over h}\Big) G(x,y)
f(y)\,d\mu_Y(y)
\end{equation}
 where $\Phi(x,y)\in C^2(I_0\times J_0;\mathbb R)$ satisfies
$\partial_{xy}^2\Phi\neq 0$ and $G(x,y)\in C^1(I_0\times J_0;\mathbb C)$.

Then there exist constants $C,\varepsilon_0>0$ such that
\begin{equation}
  \label{e:main}
\|\mathcal B_h\|_{L^2(Y,\mu_Y)\to L^2(X,\mu_X)}\leq Ch^{\varepsilon_0}.
\end{equation}
Here $\varepsilon_0$ depends only on $\delta,\delta',C_R$ as follows
\begin{equation}
\label{e:epsilon-0}
\varepsilon_0= (5C_R)^{-80\left({1\over\delta(1-\delta)}+{1\over\delta'(1-\delta')}\right)}
\end{equation}
and $C$ additionally depends on $I_0,J_0,\Phi,G$.
\end{theo}
\Remarks
1. Theorem~\ref{t:main} implies the Lebesgue measure version of the FUP,
\eqref{e:fup-simple}, with exponent $\beta={1\over 2}-\delta+\varepsilon_0$. Indeed,
assume that $(\Lambda,\mu_\Lambda)$ is $\delta$-regular up to scale~$h$ with constant~$C_R$.
Put $X:=\Lambda(h)$ and let
$\mu_X$ be $h^{\delta-1}$ times the restriction of the Lebesgue measure
to $X$. Then $(X,\mu_X)$ is $\delta$-regular up to scale $h$ with constant $30C_R^2$,
see Lemma~\ref{l:fatten-lebesgue}.
We apply Theorem~\ref{t:main} with $(Y,\mu_Y):=(X,\mu_X)$,
$G\equiv 1$, and $\Phi(x,y)=-xy$; then
$$
\|\indic_{\Lambda(h)}\mathcal F_h\indic_{\Lambda(h)}\|_{L^2(\mathbb R)\to L^2(\mathbb R)}={h^{1/2-\delta}\over\sqrt{2\pi}}\,\|\mathcal B_h\|_{L^2(X,\mu_X)\to L^2(X,\mu_X)}\leq Ch^{1/2-\delta+\varepsilon_0}.
$$

\noindent 2. Definition~\ref{d:regular-set} is slightly stronger than~\cite[Definition~1.1]{fullgap}
(where `up to scale~$h$' should be interpreted as `on scales $h$ to~1')
because it imposes an upper bound on $\mu_L(I)$ when $|I|>1$. However,
this difference is insignificant as long as $X$ is compact. Indeed, if $X\subset [-R,R]$
for some integer $R>0$, then using upper bounds on $\mu_L$ on intervals
of size 1 we get
$\mu_L(I)\leq\mu_L(X)\leq 2RC_R\leq 2RC_R|I|^\delta$
for each interval $I$ of size $|I|>1$.

\noindent 3. The restriction $\delta,\delta'>0$ is essential. Indeed, if $\delta'=0$,
$Y=\{0\}$, $\mu_Y$ is the delta measure, and $f\equiv 1$, $G\equiv 1$,
then $\|\mathcal B_h f\|_{L^2(X,\mu_X)}=\sqrt{\mu_X(X)}$. The restriction
$\delta,\delta'<1$ is technical, however in the application to Lebesgue measure
FUP this restriction is not important since $\beta={1\over 2}-\delta+\varepsilon_0<0$ when $\delta$
is close to 1.

\noindent 4. The constants in~\eqref{e:epsilon-0} are far from sharp.
However, the dependence of $\varepsilon_0$ on~$C_R$ cannot be removed entirely.
Indeed, \cite{oqm} gives examples of Cantor sets for which the best
exponent
$\varepsilon_0$ in~\eqref{e:main} decays polynomially as $C_R\to\infty$~-- see~\cite[Proposition~3.17]{oqm}
and~\S\S\ref{s:oqm-regular}--\ref{s:oqm-fup-1}.

\subsection{Ideas of the proof}

The proof of Theorem~\ref{t:main} is inspired by the method originally developed by Dolgopyat~\cite{dolgop}
and its application to essential spectral gaps for convex co-compact hyperbolic surfaces
by Naud~\cite{NaudGap}.
In fact, Theorem~\ref{t:main} implies a quantitative version of Naud's result,
see~\S\ref{s:hyper}.
More recently, Dolgopyat's method has been applied to the spectral gap problem
by Petkov--Stoyanov~\cite{PetkovStoyanov}, Stoyanov~\cite{Stoyanov1,Stoyanov2}, Oh--Winter~\cite{OhWinter},
and Magee--Oh--Winter~\cite{MOW}.

We give a sketch of the proof, assuming for simplicity that $G\equiv 1$. For $f\in L^2(Y,\mu_Y)$, we have
\begin{equation}
  \label{e:kinda-triangle}
\|\mathcal B_h f\|_{L^2(X,\mu_X)}\leq \sqrt{\mu_X(X)\mu_Y(Y)}\cdot\|f\|_{L^2(Y,\mu_Y)},
\end{equation}
applying H\"older's inequality and the
bound $\|\mathcal B_h\|_{L^1(X,\mu_X)\to L^\infty(Y,\mu_Y)}\leq 1$.
However, under a mild assumption on the differences between the phases $\Phi(x,y)$
for different $x,y$, the resulting estimate is not sharp as illustrated by the following
example where $X=Y=\{1,2\}$, $\mu_X(j)=\mu_Y(j)={1\over 2}$
for $j=1,2$, and $\omega_{j\ell}:=\Phi(j,\ell)/h$:
\begin{lemm}
  \label{l:basic-idea}
Assume that $\omega_{j\ell}\in\mathbb R$, $j,\ell=1,2$, satisfy
\begin{equation}
  \label{e:simple-condition}
\tau := \omega_{11} + \omega_{22} - \omega_{12} - \omega_{21} \notin 2\pi\mathbb Z.
\end{equation}
For $f_1,f_2\in\mathbb C$, put
$$
\begin{pmatrix}u_1 \\ u_2\end{pmatrix} := {1\over 2}
\begin{pmatrix} \exp(i\omega_{11}) & \exp(i\omega_{12}) \\
\exp(i\omega_{21}) & \exp(i\omega_{22}) \end{pmatrix}
\begin{pmatrix} f_1 \\ f_2\end{pmatrix}.
$$
Assume that $(f_1,f_2)\neq 0$. Then
\begin{equation}
  \label{e:simple-conclusion}
|u_1|^2 + |u_2|^2 <  |f_1|^2 + |f_2|^2.
\end{equation}
\end{lemm}
\Remark
Note that~\eqref{e:simple-conclusion} cannot be replaced by either of the statements
$$
|u_1| + |u_2| < |f_1| + |f_2|,\quad
\max\big(|u_1|, |u_2|\big) < \max\big(|f_1|, |f_2|\big).
$$
Indeed, the first statement fails when $f_1=0,f_2=1$. The second
one fails if $\omega_{11}=\omega_{12}$ and $f_1=f_2=1$.
This explains why we use $L^2$ norms in the iteration step, Lemma~\ref{l:main-step}.
\begin{proof}
We have
\begin{equation}
  \label{e:simple-2}
{|u_1|^2+|u_2|^2\over 2}\leq \max\big(|u_1|^2,|u_2|^2\big)\leq \Big({|f_1|+|f_2|\over 2}\Big)^2\leq {|f_1|^2+|f_2|^2\over 2}.
\end{equation}
Assume that~\eqref{e:simple-conclusion} does not hold. Then
the inequalities in~\eqref{e:simple-2} have to be equalities,
which implies that $|u_1|=|u_2|$, $|f_1|=|f_2|>0$, and for $a=1,2$
$$
\exp\big(i(\omega_{a1}-\omega_{a2})\big)f_1\bar f_2 \geq 0.
$$
The latter statement contradicts~\eqref{e:simple-condition}.
\end{proof}
To get the improvement $h^{\varepsilon_0}$ in~\eqref{e:main}, we use
non-sharpness of~\eqref{e:kinda-triangle} on many scales:
\begin{itemize}
\item We fix a large integer $L>1$ depending on $\delta,C_R$
and discretize $X$ and $Y$ on scales $1,L^{-1},\dots,L^{-K}$ where $h\sim L^{-K}$.
This results in two trees of intervals $V_X,V_Y$, with vertices of height $k$ corresponding
to intervals of length $\sim L^{-k}$.
\item For each interval $J$ in the tree $V_Y$, we consider the function
$$
F_J(x)={1\over \mu_Y(J)}\exp\Big(-{i\Phi(x,y_J)\over h}\Big)\mathcal B_h(\indic_J f)(x),
$$
where $y_J$ is the center of $J$. The function $F_J$
oscillates on scale $h/|J|$. Thus both $F_J$ and the rescaled derivative
$h|J|^{-1} F'_J$ are controlled in uniform norm by $\|f\|_{L^1(Y,\mu_Y)}$. We express this fact using the spaces $\mathcal C_\theta$
introduced in~\S\ref{s:funspace}.
\item 
If $J_1,\dots,J_B\in V_Y$ are the children of $J$, then
$F_J$ can be written as a convex combination of $F_{J_1},\dots,F_{J_B}$
multiplied by some phase factors $e^{i\Psi_b}$, see~\eqref{e:F-J-iterative}.
We then employ an iterative procedure which estimates a carefully chosen
norm of $F_J$ via the norms of $F_{J_1},\dots,F_{J_B}$. Each step in this procedure
gives a gain $1-\varepsilon_1<1$ in the norm, and after $K$ steps we obtain
a gain polynomial in $h$.
\item To obtain a gain at each step, we consider two intervals $I\in V_X$, $J\in V_Y$
such that $|I|\cdot |J|\sim Lh$, take their children $I_1,\dots,I_A$ and $J_1,\dots,J_B$,
and argue similarly to Lemma~\ref{l:basic-idea} to show that the triangle inequality
for $e^{i\Psi_1}J_1,\dots,e^{i\Psi_B}J_B$ cannot be sharp on all the intervals
$I_1,\dots,I_A$.
\item 
To do the latter, we take two pairs of children $I_a,I_a'$
(with generic points in $I_a,I_a'$ denoted $x_a,x_{a'}$)
and $J_b,J_b'$.
Due to the control on the derivatives of $F_{J_b}$,
the differences $|F_{J_b}(x_a)-F_{J_b}(x_{a'})|$ and
$|F_{J_{b'}}(x_a)-F_{J_{b'}}(x_{a'})|$ are bounded
by $(Lh)^{-1}|J|\cdot |x_a-x_{a'}|$. On the other hand, the phase shift
$\tau$ from~\eqref{e:simple-condition} equals
$$
\tau=\Psi_b(x_a)+\Psi_{b'}(x_{a'})-\Psi_{b'}(x_a)-\Psi_b(x_{a'})\sim h^{-1}(x_a-x_{a'})(y_b-y_{b'}).
$$
Choosing $a,a',b,b'$ such that $|x_a-x_{a'}|\sim L^{-2/3}|I|$, $|y_b-y_{b'}|\sim L^{-2/3}|J|$,
and recalling that $|I|\cdot |J|\sim Lh$, we see that
$\tau\sim L^{-1/3}$ does not lie in $2\pi\mathbb Z$ and it is larger
than $(Lh)^{-1}|J|\cdot|x_a-x_{a'}|\sim L^{-2/3}$. This gives the necessary improvement on each step.
Keeping track of the parameters in the argument, we obtain the bound~\eqref{e:epsilon-0} on $\varepsilon_0$.
\end{itemize}
This argument has many similarities with the method of Dolgopyat mentioned above.
In particular, an inductive argument using $L^2$ norms appears for instance
in~\cite[Lemma~5.4]{NaudGap}, which also features the spaces $\mathcal C_\theta$.
The choice of children $I_a,I_{a'},J_b,J_{b'}$ in the last step above is
similar to the non local integrability condition (NLIC), see for instance~\cite[\S\S2,5.3]{NaudGap}.
However, our inductive Lemma~\ref{l:main-step} avoids the use of Dolgopyat
operators and dense subsets (see for instance~\cite[p.138]{NaudGap}), instead
relying on strict convexity of balls in Hilbert spaces (see Lemma~\ref{l:basic-inner-product}).

Moreover, the strategy of obtaining an essential spectral gap for hyperbolic surfaces
in the present paper is significantly different from 
that of~\cite{NaudGap}. The latter uses zeta function techniques to reduce the spectral gap question
to a spectral radius bound of a Ruelle transfer operator of the Bowen--Series map associated
to the surface. The present paper instead relies on microlocal analysis of the scattering
resolvent in~\cite{hgap} to reduce the gap problem to a fractal uncertainty principle,
thus decoupling the dynamical aspects of the problem from the combinatorial ones.
The role of the group invariance of the limit set, used in~\cite{NaudGap}, is played here by its $\delta$-regularity
(proved by Sullivan~\cite{Sullivan}), and words in the group are replaced by vertices in the
discretizing tree.

\pagebreak

\subsection{Structure of the paper}

\begin{itemize}
\item In~\S\ref{s:prelim}, we establish basic properties of Ahlfors--David regular sets (\S\ref{s:discretization}),
introduce the functional spaces used (\S\ref{s:funspace}), and show several basic identities
and inequalities (\S\ref{s:basic-lemmas}).
\item In~\S\ref{s:main-proof}, we prove Theorem~\ref{t:main}.
\item In~\S\ref{s:hyper}, we apply Theorem~\ref{t:main} and the results of~\cite{hgap}
to establish an essential spectral gap for convex co-compact hyperbolic surfaces.
\item In~\S\ref{s:oqm}, we apply Theorem~\ref{t:main} and the results of~\cite{oqm,fullgap}
to establish an essential spectral gap for open quantum baker's maps.
\end{itemize}

\section{Preliminaries}
  \label{s:prelim}

\subsection{Regular sets and discretization}
  \label{s:discretization}

An \emph{interval} in $\mathbb R$ is a subset of the form
$I=[c,d]$ where $c<d$. Define the center of $I$
by $c+d\over 2$ and the size of $I$ by
$|I|=d-c$.

Let $\mu$ be a finite measure on $\mathbb R$ with compact support. 
Fix an integer $L\geq 2$.
Following~\cite[\S6.4]{hgap}, we describe the \emph{discretization
of $\mu$ with base $L$}. For each $k\in \mathbb Z$,
let $V_k$ be the set of all intervals $I=[c,d]$ which satisfy the following conditions:
\begin{itemize}
\item $c,d\in L^{-k}\mathbb Z$;
\item for each $q\in L^{-k}\mathbb Z$ with $c\leq q<d$, we have
$\mu([q,q+L^{-k}])>0$;
\item $\mu_X([c-L^{-k},c])=\mu([d,d+L^{-k}])=0$.
\end{itemize}
In other words, $V_k$ is obtained by partitioning $\mathbb R$ into intervals
of size $L^{-k}$, throwing out intervals of zero measure $\mu$, and merging
consecutive intervals.

We define the set of vertices of the discretization as
$$
V:=\bigsqcup_{k\in\mathbb Z} V_k,
$$
and define the \emph{height function}
by putting $H(I):=k$ if $I\in V_k$. (It is possible that $V_k$
intersect for different $k$, so formally speaking, a vertex
is a pair $(k,I)$ where $I\in V_k$.) We say
that $I\in V_k$ is a \emph{parent} of $I'\in V_{k+1}$,
and $I'$ is a \emph{child} of $I$, if $I'\subset I$.
It is easy to check that the resulting structure has the following properties:
\begin{itemize}
\item any two distinct intervals $I,I'\in V_k$ are
at least $L^{-k}$ apart;
\item $\mu(\mathbb R\setminus \bigsqcup_{I\in V_k} I)=0$ for all $k$;
\item each $I\in V_k$ has exactly one parent;
\item if $I\in V_k$ and $I_1,\dots,I_n\in V_{k+1}$ are the children of $I$,
then
\begin{equation}
  \label{e:child-convex}
0<\mu(I)=\sum_{j=1}^n \mu(I_j).
\end{equation}
\end{itemize}
For regular sets, the discretization has the following additional properties:
\begin{lemm}
  \label{l:regtree}
Let $L\geq 2$, $K>0$ be integers and
assume that $(X,\mu_X)$ is $\delta$-regular up to scale~$L^{-K}$ with regularity constant $C_R$,
where $0<\delta<1$.
Then the discretization of~$\mu_X$ with base $L$ has the following properties:

\noindent 1. Each $I\in V$ with $0\leq H(I)\leq K$ satisfies
for $C_R':=(3C_R^2)^{1\over 1-\delta}$,
\begin{gather}
  \label{e:regtree-1}
L^{-H(I)}\leq |I|\leq C_R' L^{-H(I)},\\
  \label{e:regtree-2}
C_R^{-1}L^{-\delta H(I)}\leq \mu_X(I)\leq C_R (C_R')^\delta L^{-\delta H(I)}.
\end{gather}

\noindent 2. If $I'$ is a child of $I\in V$ and $0\leq H(I)<K$, then
\begin{equation}
  \label{e:convex-coeff-bound}
{\mu_X(I')\over\mu_X(I)}\geq {L^{-\delta}\over C'_R}.
\end{equation}

\noindent 3. Assume that
\begin{equation}
  \label{e:tree-L-bound}
L\geq (4C_R)^{6\over \delta(1-\delta)}.
\end{equation}
Then for each $I\in V$ with $0\leq H(I)<K$, there exist two children
$I',I''$ of $I$ such that
$$
{1\over 2}C_R^{-2/\delta} L^{-H(I)-2/3}\leq |x'-x''|\leq 2 L^{-H(I)-2/3}\quad\text{for all }
x'\in I',\
x''\in I''.
$$
\end{lemm}
\Remark Parts 1 and~2 of the lemma state that the tree of intervals
discretizing $\mu_X$ is approximately regular. Part~3, which is used
at the end of~\S\ref{s:main-step}, states that once the base of discretization
$L$ is large enough, each interval $I$ in the tree has two children
which are $\sim L^{-H(I)-2/3}$ apart from each other. A similar statement
would hold if $2/3$ were replaced by any number in $(0,1)$.
\begin{proof}
1. Put $k:=H(I)$. The lower bound on $|I|$ follows from the construction of the
discretization. To show the upper bound, assume that
$I=[c,d]$ and $d-c=ML^{-k}$. For each
$q\in L^{-k}\mathbb Z$ with $c\leq q<d$, we have $\mu_X([q,q+L^{-k}])>0$,
thus there exists $x_q\in [q,q+L^{-k}]\cap X$. Let
$I_q$ be the interval of size $L^{-k}$ centered at $x_q$,
see Figure~\ref{f:deltareg}.
Then
$$
\mu_X\Big(\bigcup\nolimits_q I_q\Big)\leq 
\mu_X\big([c-L^{-k},d+L^{-k}]\big)=
\mu_X(I)\leq C_R(ML^{-k})^\delta.
$$
On the other hand, each point is covered by at most 3 intervals $I_q$, therefore
$$
MC_R^{-1}L^{-k\delta}\leq \sum_q \mu_X(I_q)\leq 3\mu_X\Big(\bigcup\nolimits_q I_q\Big).
$$
Together these two inequalities imply that $M\leq C'_R$, giving~\eqref{e:regtree-1}.

The upper bound on $\mu_X(I)$ follows from~\eqref{e:regtree-1}.
To show the lower bound, take $x\in I\cap X$ and let $I'$ be the interval
of size $L^{-k}$ centered at $x$. Then $\mu_X(I'\setminus I)=0$,
therefore $\mu_X(I)\geq \mu_X(I')\geq C_R^{-1}L^{-\delta k}$.

\noindent 2. This follows directly from~\eqref{e:regtree-2} and the fact that
$C_R^2(C'_R)^\delta\leq C'_R$.

\begin{figure}
\includegraphics{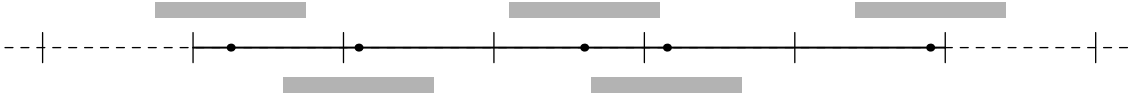}
\caption{An illustration of the proof of the upper bound in~\eqref{e:regtree-1}.
The ticks mark points in $L^{-k}\mathbb Z$, the solid interval is $I$,
the dots mark the points $x_q$, and the shaded intervals are $I_q$.
The intervals of length $L^{-k}$ adjacent to $I$ have zero measure $\mu_X$.
}
\label{f:deltareg}
\end{figure}

\noindent 3. Put $k:=H(I)$. Take $x\in I\cap X$ and let $J$ be the interval of size $L^{-k-2/3}$ centered
at $x$.
Let $I_1,\dots,I_n$ be all the intervals in $V_{k+1}$ which intersect~$J$; they all have to be children of $I$. Let $x_1,\dots,x_n$ be the centers of $I_1,\dots,I_n$.
Define
$$
T:=L^{k+2/3}\max_{j,\ell}|x_j-x_\ell|.
$$
By~\eqref{e:regtree-1}, we have $|I_j|\leq C'_RL^{-k-1}$ and thus $T\leq 1+C_R'L^{-1/3}$.
On the other hand, the union of $I_1,\dots,I_n$ is contained
in an interval of size $TL^{-k-2/3}+C'_RL^{-k-1}$. Therefore
$$
C_R^{-1}L^{-\delta(k+2/3)}\leq \mu_X(J)\leq \sum_{j=1}^n \mu_X(I_j)\leq C_R(TL^{-k-2/3}+C'_RL^{-k-1})^\delta.
$$
This implies that $T\geq C_R^{-2/\delta}-C'_R L^{-1/3}$.

Now, put $I':=I_j$, $I''=I_\ell$ where $j,\ell$ are chosen so
that $T=L^{k+2/3}|x_j-x_\ell|$. Then for each $x'\in I',x''\in I''$,
we have by~\eqref{e:tree-L-bound}
$$
{1\over 2}C_R^{-2/\delta}\leq
C_R^{-2/\delta}-2C'_R L^{-1/3}
\leq L^{k+2/3}|x'-x''|\leq 1+2C'_R L^{-1/3}\leq 2
$$
which finishes the proof.
\end{proof}
We finally the following estimates on Lebesgue measure of neighborhoods
of $\delta$-regular set which are used in~\S\S\ref{s:hyper},\ref{s:oqm}:
\begin{lemm}
  \label{l:fatten-lebesgue}
Assume that $(\Lambda,\mu_\Lambda)$ is $\delta$-regular up to scale $h\in (0,1)$
with constant $C_R$. Let $X:=\Lambda(h)=\Lambda+[-h,h]$ be the $h$-neighborhood of $\Lambda$
and define the measure $\mu_X$ by
\begin{equation}
  \label{e:fatten-lebesgue}
\mu_X(A):=h^{\delta-1}\mu_L(X\cap A),\quad
A\subset\mathbb R,
\end{equation}
where $\mu_L$ denotes the Lebesgue measure. Then $(X,\mu_X)$ is $\delta$-regular
up to scale $h$ with constant $C'_R:=30C_R^2$.
\end{lemm}
\begin{proof}
We follow~\cite[Lemma~7.4]{hgap}. Let $I\subset \mathbb R$ be an interval with
$|I|\geq h$. Let $x_1,\dots,x_N\in \Lambda\cap I(h)$ be a maximal set of $2h$-separated points.
Denote by $I'_n$ the interval of size $h$ centered at $x_n$. Since $I'_n$ are disjoint
and their union is contained in $I(2h)$, which is an interval of size
$|I|+4h\leq 5|I|$, we have
\begin{equation}
  \label{e:fatleb-1}
N\cdot C_R^{-1}h^\delta \leq \sum_{n=1}^N \mu_\Lambda(I'_n)\leq \mu_\Lambda(I(2h))\leq 5C_R|I|^\delta.
\end{equation}
Next, let $I_n$ be the interval of size $6h$ centered at $x_n$. Then
$X\cap I$ is contained in the union of $I_n$ and thus
\begin{equation}
  \label{e:fatleb-2}
\mu_L(X\cap I)\leq \sum_{n=1}^N \mu_L(I_n)=6hN. 
\end{equation}
Together \eqref{e:fatleb-1} and~\eqref{e:fatleb-2} give the required upper bound
$$
\mu_X(I)=h^{\delta-1}\mu_L(X\cap I)\leq 30 C_R^2|I|^\delta. 
$$
Now, assume additionally that $|I|\leq 1$ and $I$ is centered at a point in $X$. Let $y_1,\dots,y_M\in\Lambda\cap I$
be a maximal set of $h$-separated points. Denote by $I_m$ the interval of size $2h$ centered
at $y_m$. Then $\Lambda\cap I$ is contained in the union of $I_m$, therefore
\begin{equation}
  \label{e:fatleb-3}
C_R^{-1}|I|^\delta\leq \mu_\Lambda(I)=\mu_\Lambda(\Lambda\cap I)\leq \sum_{m=1}^M \mu_\Lambda(I_m)\leq M\cdot 2C_R h^\delta.
\end{equation}
Next, let $I'_m$ be the interval of size $h$ centered at $y_m$. Then $I'_m\subset X$ are nonoverlapping
and each $I'_m\cap I$ has size at least $h/2$, therefore
\begin{equation}
  \label{e:fatleb-4}
\mu_L(X\cap I)\geq \sum_{m=1}^M \mu_L(I'_m\cap I)\geq Mh/2.
\end{equation}
Combining~\eqref{e:fatleb-3} and~\eqref{e:fatleb-4} gives the required lower bound
$$
\mu_X(I)=h^{\delta-1}\mu_L(X\cap I)\geq {1\over 4C_R^2}|I|^\delta.
$$
and finishes the proof.
\end{proof}

\subsection{Functional spaces}
  \label{s:funspace}

For a constant $\theta>0$ and an interval $I$, let
$\mathcal C_\theta(I)$ be the space $C^1(I)$ with the norm
$$
\|f\|_{\mathcal C_\theta(I)}:=\max\big(\sup_I |f|,\theta |I|\cdot \sup_I |f'|\big).
$$
The following lemma shows that multiplications by functions
of the form $\exp(i\psi)$ have norm 1 when mapping $\mathcal C_\theta(I)$ into the corresponding
space for a sufficiently small subinterval of $I$:
\begin{lemm}
  \label{l:derbound}
Consider intervals
\begin{equation}
  \label{e:derbound-assumption-1}
I'\ \subset\ I,\quad
|I'|\leq {|I|\over 4}.
\end{equation}
Assume that
$\psi\in C^\infty(I;\mathbb R)$ and $\theta>0$ are such that
\begin{equation}
  \label{e:derbound-assumption-2}
4\theta|I'|\cdot\sup_{I'}|\psi'|\leq 1.
\end{equation}
Then for each $f\in \mathcal C_\theta(I)$, we have
$\|\exp(i\psi)f\|_{\mathcal C_\theta(I')}\leq \|f\|_{\mathcal C_\theta(I)}$ and
\begin{equation}
  \label{e:derbound}
\theta|I'|\cdot \sup_{I'}\big|(\exp(i\psi)f)'\big|\ \leq\ 
{\|f\|_{\mathcal C_\theta(I)}\over 2}.
\end{equation}
\end{lemm}
\begin{proof}
The left-hand side of~\eqref{e:derbound} is bounded from above by
$$
\theta|I'|\cdot \big(\sup_{I'}|\psi' f|+\sup_{I'}|f'|\big).
$$
From~\eqref{e:derbound-assumption-2},
\eqref{e:derbound-assumption-1} we get
$$
\theta|I'|\cdot\sup_{I'}|\psi'f|\leq {\|f\|_{\mathcal C_\theta(I)}\over 4},\quad
\theta|I'|\cdot \sup_{I'}|f'|\leq {\|f\|_{\mathcal C_\theta(I)}\over 4}
$$
which finishes the proof of~\eqref{e:derbound}.
The bound~\eqref{e:derbound} implies that $\|\exp(i\psi)f\|_{\mathcal C_\theta(I')}\leq \|f\|_{\mathcal C_\theta(I)}$.
\end{proof}
The following is a direct consequence of the mean value theorem:
\begin{lemm}
  \label{l:mvt-theta}
Let $f\in \mathcal C_\theta(I)$. Then for all $x,x'\in I$,
we have
\begin{equation}
  \label{e:mvt-theta}
|f(x)-f(x')|\leq {|x-x'|\over \theta|I|}\cdot \|f\|_{\mathcal C_\theta(I)}.
\end{equation}
\end{lemm}

\subsection{A few technical lemmas}
\label{s:basic-lemmas}

The following is a two-dimensional analogue of the mean value theorem:
\begin{lemm}
  \label{l:mean-value}
Let $I=[c_1,d_1]$ and $J=[c_2,d_2]$ be two intervals and
$\Phi\in C^2(I\times J;\mathbb R)$. Then there exists
$(x_0,y_0)\in I\times J$ such that
$$
\Phi(c_1,c_2)+\Phi(d_1,d_2)-\Phi(c_1,d_2)-\Phi(d_1,c_2)
=|I|\cdot|J|\cdot \partial^2_{xy}\Phi(x_0,y_0).
$$
\end{lemm}
\begin{proof}
Replacing $\Phi(x,y)$ by $\Phi(x,y)-\Phi(c_1,y)-\Phi(x,c_2)+\Phi(c_1,c_2)$,
we may assume that $\Phi(c_1,y)=0$ and $\Phi(x,c_2)=0$ for
all $x\in I$, $y\in J$. By the mean value theorem,
we have $\Phi(d_1,d_2)=|I|\cdot \partial_x\Phi(x_0,d_2)$ for some $x_0\in I$.
Applying the mean value theorem again, we have
$\partial_x\Phi(x_0,d_2)=|J|\cdot \partial^2_{xy}\Phi(x_0,y_0)$
for some $y_0\in J$, finishing the proof.
\end{proof}

\begin{lemm}
  \label{l:silly}
Assume that $\tau\in\mathbb R$ and $|\tau|\leq \pi$. Then
$|e^{i\tau}-1|\geq {2\over\pi}|\tau|$.
\end{lemm}
\begin{proof}
We have $|e^{i\tau}-1|=2\sin(|\tau|/2)$. It remains
to use that $\sin x\geq {2\over\pi}x$ when $0\leq x\leq{\pi\over 2}$,
which follows from the concavity of $\sin x$ on that interval.
\end{proof}
The next lemma, used several times in~\S\ref{s:main-step},
is a quantitative version of the fact that balls in Hilbert spaces
are strictly convex:
\begin{lemm}
  \label{l:basic-inner-product}
Assume that $\mathcal H$ is a Hilbert space, $f_1,\dots,f_n\in \mathcal H$,
$p_1,\dots,p_n\geq 0$, and $p_1+\dots+p_n=1$. Then
\begin{equation}
  \label{e:basic-inner-product}
\Big\|\sum_{j=1}^n p_j f_j\Big\|_{\mathcal H}^2
=\sum_{j=1}^n p_j \|f_j\|_{\mathcal H}^2
-\sum_{1\leq j<\ell\leq n} p_jp_\ell \|f_j-f_\ell\|_{\mathcal H}^2.
\end{equation}
If moreover for some $\varepsilon,R\geq 0$
\begin{equation}
  \label{e:extreme-l2-cond}
\sum_{j=1}^n p_j \|f_j\|_{\mathcal H}^2=R,\quad
\Big\|\sum_{j=1}^n p_jf_j\Big\|_{\mathcal H}^2\geq (1-\varepsilon) R,\quad
p_{\min}:=\min_j p_j\geq 2\sqrt{\varepsilon}
\end{equation}
then for all $j$
\begin{equation}
  \label{e:extreme-l2}
{\sqrt R\over 2}\leq \|f_j\|_{\mathcal H}\leq 2\sqrt{R}.
\end{equation}
\end{lemm}
\begin{proof}
The identity~\eqref{e:basic-inner-product} follows by a direct computation.
To show~\eqref{e:extreme-l2}, note that by~\eqref{e:basic-inner-product}
and~\eqref{e:extreme-l2-cond} for each $j,\ell$
$$
\|f_j-f_\ell\|_{\mathcal H}^2\leq {\varepsilon R\over p_{\min}^2}\leq{R\over 4}.
$$
Put $f_{\max}:=\max_j \|f_j\|_{\mathcal H}$ and $f_{\min}:=\min_j \|f_j\|_{\mathcal H}$, then
$$
f_{\max}-f_{\min}\leq{\sqrt{R}\over 2},\quad
f_{\min}\leq \sqrt{R}\leq f_{\max}
$$
which implies~\eqref{e:extreme-l2}.
\end{proof}

\begin{lemm}
  \label{l:extreme-l1}
Assume that $\alpha_j,p_j\geq 0$, $j=1,\dots,n$, $p_1+\dots+p_n=1$,
and for some $\varepsilon,R\geq 0$
$$
\sum_{j=1}^n p_j\alpha_j \geq (1-\varepsilon) R,\quad
\max_j \alpha_j\leq R,\quad
p_{\min}:=\min_j p_j\geq 2\varepsilon.
$$
Then for all $j$,
$$
\alpha_j\geq {R\over 2}.
$$
\end{lemm}
\begin{proof}
We have
$$
\sum_{j=1}^n p_j(R-\alpha_j)\leq \varepsilon R.
$$
All the terms in the sum are nonnegative, therefore for all $j$
$$
R-\alpha_j\leq {\varepsilon R\over p_{\min}}\leq {R\over 2}
$$
finishing the proof.
\end{proof}

\section{Proof of Theorem~\texorpdfstring{\ref{t:main}}{1}}
  \label{s:main-proof}

\subsection{The iterative argument}

In this section, we prove the following statement
which can be viewed as a special case of Theorem~\ref{t:main}.
Its proof
relies on an inductive bound, Lemma~\ref{l:main-step},
which is proved in~\S\ref{s:main-step}.
In~\S\ref{s:reduce-basic},
we deduce Theorem~\ref{t:main} from Proposition~\ref{l:almost-main},
in particular removing the condition~\eqref{e:C-Phi}.
\begin{prop}
  \label{l:almost-main}
Let $\delta,\delta'\in (0,1)$, $C_R>1$, $I_0,J_0\subset \mathbb R$ be some intervals,
$G\in C^1(I_0\times J_0;\mathbb C)$, and the phase function
$\Phi\in C^2(I_0\times J_0;\mathbb R)$ satisfy
\begin{equation}
  \label{e:C-Phi}
{1\over 2}<|\partial^2_{xy}\Phi(x,y)|<2\quad\text{for all }
(x,y)\in I_0\times J_0.
\end{equation}
Choose constants $C'_R>0$ and $L\in\mathbb N$ such that
\begin{equation}
  \label{e:L-restriction}
C'_R=(2C_R)^{2\over 1-\max(\delta,\delta')},\quad
L\geq\big(2 C'_R(6C_R)^{{1\over\delta}+{1\over\delta'}}\big)^6.
\end{equation}
Fix $K_0\in\mathbb N_0$ and put $h:=L^{-K}$ for some $K\in\mathbb N_0$, $K\geq 2K_0$.
Assume that
$(X,\mu_X)$ is $\delta$-regular, and $(Y,\mu_Y)$ is $\delta'$-regular,
up to scale $L^{K_0-K}$ with regularity constant~$C_R$,
and $X\subset I_0$, $Y\subset J_0$.
Put
\begin{equation}
  \label{e:epsilon-1}
\varepsilon_1:=
10^{-5}\big(C_R^{{1\over \delta}+{1\over\delta'}}C'_R\big)^{-4}
L^{-5},\quad
\varepsilon_0:=-{\log(1-\varepsilon_1)\over2\log L}.
\end{equation}
Then for some $C$ depending only on $K_0,G,\mu_X(X),\mu_Y(Y)$,
and $\mathcal B_h$ defined in~\eqref{e:B-h},
\begin{equation}
  \label{e:almost-main}
\|\mathcal B_h f\|_{L^2(X,\mu_X)}\leq Ch^{\varepsilon_0}
\|f\|_{L^2(Y,\mu_Y)}\quad\text{for all }f\in L^2(Y,\mu_Y).
\end{equation}
\end{prop}
\Remark Proposition~\ref{l:almost-main} has complicated hypotheses in order to make it
useful for the proof of Theorem~\ref{t:main}. However, the argument is
essentially the same in the following special case which could simplify
the reading of the proof below: $\delta=\delta'$, $G\equiv 1$, $\Phi(x,y)=xy$,
$K_0=0$. Note that in this case $\mathcal B_h$ is related
to the semiclassical Fourier transform~\eqref{e:F-h}.

To start the proof of Proposition~\ref{l:almost-main},
we extend $\Phi$ to a function in $C^2(\mathbb R^2;\mathbb R)$ such that~\eqref{e:C-Phi} still holds,
and extend $G$ to a function in $C^1(\mathbb R^2;\mathbb C)$ such that $G,\partial_xG$ are uniformly bounded.
Following~\S\ref{s:discretization} consider the discretizations
of~$\mu_X,\mu_Y$ with base~$L$, denoting by~$V_X,V_Y$ the sets of
vertices and by~$H$ the height functions.

Fix $f\in L^2(Y,\mu_Y)$. For each $J\in V_Y$, 
let $y_J$ denote the center of $J$ and define the function of
$x\in \mathbb R$,
\begin{equation}
  \label{e:F-J}
F_J(x)={1\over\mu_Y(J)}\int_{J}\exp\bigg({i\big(\Phi(x,y)-\Phi(x,y_J)\big)\over h}\bigg)
G(x,y)f(y)\,d\mu_Y(y).
\end{equation}
In terms of the operator $\mathcal B_h$ from~\eqref{e:B-h}, we may write
\begin{equation}
  \label{e:F-J-2}
F_J(x)={1\over\mu_Y(J)}\exp\Big(-{i\Phi(x,y_J)\over h}\Big) \mathcal B_h(\indic_J f)(x).
\end{equation}
Put
\begin{equation}
  \label{e:theta-restriction}
\theta:={1\over 8 (C'_R)^2}
\end{equation}
and for $J\in V_Y$ define the piecewise constant function
$E_J\in L^\infty(X,\mu_X)$ using
the space $\mathcal C_\theta(I)$ defined in~\S\ref{s:funspace}:
\begin{equation}
  \label{e:E-J}
E_J(x)=\|F_J\|_{\mathcal C_\theta(I)}
\quad\text{where}\quad
x\in I\in V_X,\
H(I)+H(J)=K.
\end{equation}
See Figure~\ref{f:grids}.
Note that $|F_J(x)|\leq E_J(x)$ for $\mu_X$-almost every $x$.
\begin{figure}
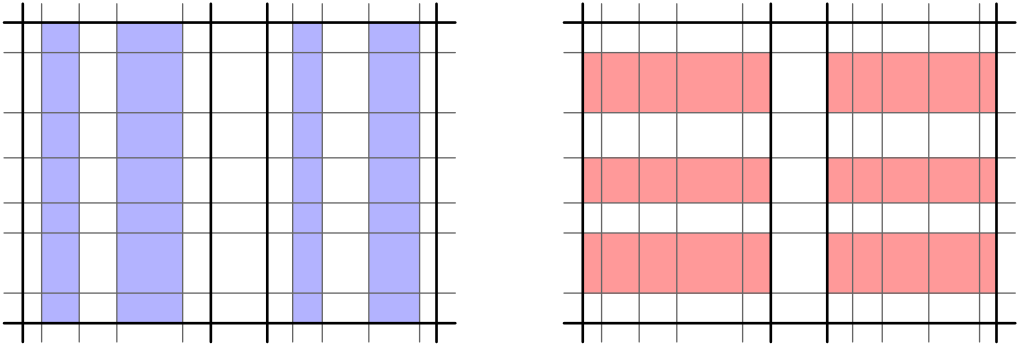

\includegraphics{regfup.2}
\qquad\quad
\includegraphics{regfup.3}
\caption{An illustration of~\eqref{e:E-J} in the case $K=1$. The vertical lines
mark the endpoints of intervals in $V_X$ and the horizontal lines,
the endpoints of intervals in $V_Y$. The thick lines correspond
to intervals of height 0 and the thin lines, to intervals of height 1.
The shaded rectangles have the form $I\times J$,
$I\in V_X,J\in V_Y$, where $E_J$ is constant on $I$, and 
the shaded rectangles on the left/on the right correspond
to the left/right hand sides of~\eqref{e:main-step} for $H(J)=0$.}
\label{f:grids}
\end{figure}

The $L^2$ norms of the functions $E_J$ satisfy the following key bound,
proved in~\S\ref{s:main-step}, which gives
an improvement from one scale to the next.
The use of the $L^2$ norm of $E_J$ as the monotone quantity is convenient
for several reasons. On one hand, the averaging provided by the $L^2$ norm
means it is only necessary to show an improvement on $F_J$ in sufficiently
many places; more precisely we will show in~\eqref{e:boston-1}
that such improvement happens on at least one child of each interval
$I\in V(X)$ with $H(I)+H(J)=K-1$.
On the other hand, such improvement is obtained by a pointwise argument which also
uses that $F_{J_b}$ are slowly varying on each interval $I$ with $H(I)+H(J)=K-1$ (see Lemma~\ref{l:FAB-upper});
this motivates the use of $\mathcal C_\theta(I)$ norms in the definition of $E_{J_b}$.
\begin{lemm}
  \label{l:main-step}
Let $J\in V_Y$
with $K_0\leq H(J)<K-K_0$ and $J_1,\dots,J_B\in V_Y$ be the children of $J$.
Then, with $\varepsilon_1$ defined in~\eqref{e:epsilon-1},
\begin{equation}
  \label{e:main-step}
\|E_J\|_{L^2(X,\mu_X)}^2\leq (1-\varepsilon_1)\sum_{b=1}^B {\mu_Y(J_b)\over\mu_Y(J)}
\|E_{J_b}\|_{L^2(X,\mu_X)}^2.
\end{equation}
\end{lemm}
Iterating Lemma~\ref{l:main-step}, we obtain
\begin{proof}[Proof of Proposition~\ref{l:almost-main}]
First of all, we show that for all $J\in V_Y$ with $H(J)=K-K_0$,
and some constant $C_0$ depending on $G,\mu_X(X)$ and defined below
\begin{equation}
  \label{e:start-iteration}
\|E_J\|_{L^2(X,\mu_X)}^2\leq C_0\,{ \|f\|_{L^2(J,\mu_Y)}^2\over \mu_Y(J)}.
\end{equation}
Indeed, take $I\in V_X$ such that $H(I)=K_0$.
By~\eqref{e:regtree-1} and~\eqref{e:C-Phi}, for all $y\in J$
$$
{1\over h}\sup_{x\in I}|\partial_x\Phi(x,y)-\partial_x\Phi(x,y_J)|\leq
{2\over h}|J|\leq 2C'_RL^{K_0}
$$
and thus by~\eqref{e:regtree-1} and~\eqref{e:theta-restriction}
$$
{4\theta|I|\over h}\sup_{x\in I}|\partial_x\Phi(x,y)-\partial_x\Phi(x,y_J)|\leq 1.
$$
Arguing similarly to Lemma~\ref{l:derbound}, we obtain for all $y\in J$
$$
\Big\|\exp\Big({i\big(\Phi(x,y)-\Phi(x,y_J)\big)\over h}\Big)G(x,y)\Big\|_{\mathcal C_\theta(I)}\leq C_G:=\max\big(\sup |G|,\sup|\partial_x G|\big).
$$
Using H\"older's inequality in~\eqref{e:F-J}, we obtain
$$
E_J|_I=\|F_J\|_{\mathcal C_\theta(I)}\leq {C_G\over\mu_Y(J)}\int_J |f(y)|\,d\mu_Y(y)
\leq {C_G\|f\|_{L^2(J,\mu_Y)}\over\sqrt{\mu_Y(J)}}
$$
and~\eqref{e:start-iteration} follows by integration in $x$,
where we put $C_0:=C^2_G\mu_X(X)$.

Now, arguing by induction on $H(J)$ with~\eqref{e:start-iteration} as base
and~\eqref{e:main-step} as inductive step, we obtain
for all $J\in V_Y$ with $K_0\leq H(J)\leq K-K_0$,
$$
\|E_J\|_{L^2(X,\mu_X)}^2\leq C_0(1-\varepsilon_1)^{K-K_0-H(J)}\,{ \|f\|_{L^2(J,\mu_Y)}^2\over \mu_Y(J)}.
$$
In particular, for all $J\in V_Y$ with $H(J)=K_0$, we have by~\eqref{e:F-J-2}
$$
\Big\|{\mathcal B_h(\indic_J f)\over\mu_Y(J)}\Big\|_{L^2(X,\mu_X)}^2
=\|F_J\|_{L^2(X,\mu_X)}^2
\leq \|E_J\|_{L^2(X,\mu_X)}^2\leq C_1h^{2\varepsilon_0}{\|f\|_{L^2(J,\mu_Y)}^2\over \mu_Y(J)}.
$$
where $C_1:=C_0(1-\varepsilon_1)^{-2K_0}$.
Using the identity
$$
\mathcal B_hf=\mu_Y(Y)\sum_{J\in V_Y,\, H(J)=K_0} {\mu_Y(J)\over\mu_Y(Y)}\cdot {\mathcal B_h(\indic_J f)\over \mu_Y(J)}
$$
and~\eqref{e:basic-inner-product}, we estimate
$$
\|\mathcal B_h f\|_{L^2(X,\mu_X)}^2\leq C_1\mu_Y(Y)h^{2\varepsilon_0}\|f\|_{L^2(Y,\mu_Y)}^2
$$
and~\eqref{e:almost-main} follows with $C:=C_G(1-\varepsilon_1)^{-K_0}\sqrt{\mu_X(X)\mu_Y(Y)}$.
\end{proof}

\subsection{The inductive step}
\label{s:main-step}

In this section we prove Lemma~\ref{l:main-step}.
Let $J\in V_Y$ satisfy $K_0\leq H(J)<K-K_0$
and $J_1,\dots,J_B$ be the children of $J$.
It suffices to show that for all $I\in V_X$ with $H(I)+H(J)=K-1$ we have
\begin{equation}
  \label{e:boston-1}
\|E_J\|_{L^2(I,\mu_X)}^2\leq (1-\varepsilon_1)\sum_{b=1}^B{\mu_Y(J_b)\over\mu_Y(J)}\|E_{J_b}\|_{L^2(I,\mu_X)}^2.
\end{equation}
Indeed, summing~\eqref{e:boston-1} over $I$, we obtain~\eqref{e:main-step}.

Fix $I\in V_X$ with $H(I)+H(J)=K-1$ and let $I_1,\dots,I_A$ be the children of $I$.
Define
$$
p_a:={\mu_X(I_a)\over\mu_X(I)},\quad
q_b:={\mu_Y(J_b)\over \mu_Y(J)}.
$$
Note that $p_a,q_b\geq 0$ and $p_1+\dots+p_A=q_1+\dots+q_B=1$.

The functions $F_J$ and $F_{J_b}$ are related by the following formula:
\begin{equation}
  \label{e:F-J-iterative}
F_J=\sum_{b=1}^B q_b \exp(i\Psi_b) F_{J_b},\quad
\Psi_b(x):={\Phi(x,y_{J_b})-\Phi(x,y_J)\over h}.
\end{equation}
That is, $F_J$ is a convex combination of $F_{J_1},\dots,F_{J_B}$ multiplied
by the phase factors $\exp(i\Psi_b)$. At the end of this subsection
we exploit cancellation between these phase factors to show~\eqref{e:boston-1}.
However there are several preparatory steps necessary.
Before we proceed with the proof, we show the version
of~\eqref{e:boston-1} with no improvement:
\begin{lemm}
We have
\begin{equation}
  \label{e:ideas-bound}
\|E_J\|_{L^2(I,\mu_X)}^2\leq \sum_{b=1}^B q_b\|E_{J_b}\|_{L^2(I,\mu_X)}^2.
\end{equation}
\end{lemm}
\begin{proof}
By~\eqref{e:regtree-1}, \eqref{e:C-Phi}, and~\eqref{e:theta-restriction}, we have
for all $a,b$
\begin{equation}
  \label{e:derbound-iteration}
4\theta |I_a|\cdot \sup_{I}|\Psi'_b|
\leq {8\theta|I_a|\cdot |J|\over h}
\leq 1.
\end{equation}
Moreover, by~\eqref{e:regtree-1} and~\eqref{e:L-restriction} we have
$|I_a|\leq {1\over 4}|I|$.
Applying Lemma~\ref{l:derbound}, we obtain
$$
\|\exp(i\Psi_b)F_{J_b}\|_{C_\theta(I_a)}\leq \|F_{J_b}\|_{C_\theta(I)}.
$$
By~\eqref{e:F-J-iterative} and~\eqref{e:basic-inner-product} we then have
\begin{equation}
  \label{e:ideas-bound-1}
\|F_J\|_{C_\theta(I_a)}^2\leq \bigg(\sum_{b=1}^B q_b\|F_{J_b}\|_{C_\theta(I)}\bigg)^2
\leq \sum_{b=1}^B q_b\|F_{J_b}\|_{C_\theta(I)}^2.
\end{equation}
By~\eqref{e:E-J}, we have for all $a,b$
\begin{equation}
  \label{e:step-by-step}
E_J|_{I_a}=\|F_J\|_{\mathcal C_\theta(I_a)},\quad
E_{J_b}|_I=\|F_{J_b}\|_{\mathcal C_\theta(I)}.
\end{equation}
Now, summing both sides of~\eqref{e:ideas-bound-1} over $a$ with weights $\mu_X(I_a)$, we obtain~\eqref{e:ideas-bound}.
\end{proof}
The rest of this section is dedicated to the proof of~\eqref{e:boston-1},
studying the situations in which the bound~\eqref{e:ideas-bound} is almost sharp
and ultimately reaching a contradiction.
The argument is similar in spirit to Lemma~\ref{l:basic-idea}.
In fact we can view Lemma~\ref{l:basic-idea} as the special degenerate case
when $A=B=2$, $p_a=q_b={1\over 2}$, the intervals $I_a$ are replaced
by points $x_a$, $F_{J_b}\equiv f_b$ are constants,
$u_a=F_J(x_a)$,
and $\omega_{ab}=\Psi_b(x_a)$.
The general case is more technically complicated. In particular we use
Lemma~\ref{l:basic-inner-product} to deal with general convex combinations.
We also use $\delta$-regularity in many places,
for instance to show that the coefficients $p_a,q_b$ are bounded away from zero
and to get the phase factor cancellations in~\eqref{e:NLIC-gain} at the end of the proof.
The reading of the argument below may be simplified by making
the illegal choice $\varepsilon_1:=0$.

We henceforth assume that~\eqref{e:boston-1} does not hold. Put
\begin{equation}
  \label{e:R-def}
R:=\sum_{b=1}^B q_b\|F_{J_b}\|^2_{\mathcal C_\theta(I)}. 
\end{equation}
By~\eqref{e:step-by-step}, the failure of~\eqref{e:boston-1} can be rewritten as
\begin{equation}
  \label{e:contrastep-2}
\sum_{a=1}^A p_a \|F_J\|^2_{\mathcal C_\theta(I_a)} >
(1-\varepsilon_1)R.
\end{equation}
We note for future use that $p_a,q_b$
are bounded below by~\eqref{e:convex-coeff-bound}:
\begin{equation}
  \label{e:pq-bound}
p_{\min}:=\min_a p_a\geq {L^{-\delta}\over C'_R},\quad
q_{\min}:=\min_b q_b\geq {L^{-\delta'}\over C'_R}.
\end{equation}

We first deduce from~\eqref{e:R-def} and the smallness of $\varepsilon_1$
an upper bound on each $\|F_{J_b}\|_{\mathcal C_\theta(I)}$
in terms of the averaged quantity $R$:
\begin{lemm}
  \label{l:F-J-b-lower}
We have for all $b$,
\begin{equation}
  \label{e:F-J-b-lower}
\|F_{J_b}\|_{\mathcal C_\theta(I)}\leq 2\sqrt{R}. 
\end{equation}
\end{lemm}
\begin{proof}
The first inequality in~\eqref{e:ideas-bound-1} together with~\eqref{e:contrastep-2} implies
\begin{equation}
  \label{e:lower-mink}
\Big(\sum_{b=1}^B q_b\|F_{J_b}\|_{\mathcal C_\theta(I)}\Big)^2\geq 
\sum_{a=1}^A p_a\|F_J\|_{\mathcal C_\theta(I_a)}^2
\geq (1-\varepsilon_1)R.
\end{equation}
By~\eqref{e:epsilon-1} and~\eqref{e:pq-bound} we have
$q_{\min}\geq 2\sqrt{\varepsilon_1}$.
Applying~\eqref{e:extreme-l2} to $f_b:=\|F_{J_b}\|_{\mathcal C_{\theta(I)}}$ with~\eqref{e:R-def} and~\eqref{e:lower-mink},
we obtain~\eqref{e:F-J-b-lower}.
\end{proof}
We next obtain a version of~\eqref{e:contrastep-2} which gives a lower bound on
the size of $F_J$, rather than on the norm
$\|F_J\|_{\mathcal C_\theta(I_a)}$:
\begin{lemm}
  \label{l:get-rid-of-ders}
There exist $x_a\in I_a$, $a=1,\dots,A$,
such that
\begin{equation}
  \label{e:contrastep-3}
\sum_{a=1}^A p_a |F_J(x_a)|^2 > (1-2\varepsilon_1)R.
\end{equation}
\end{lemm}
\begin{proof}
By Lemma~\ref{l:derbound} and~\eqref{e:derbound-iteration}, we have
$$
\theta|I_a|\cdot \sup_{I_a}\big|(\exp(i\Psi_b) F_{J_b})'\big|\leq {\|F_{J_b}\|_{\mathcal C_\theta(I)}\over 2}.
$$
It follows by~\eqref{e:F-J-iterative} and the triangle inequality that for all $a$,
\begin{equation}
  \label{e:panda-1}
\|F_J\|_{\mathcal C_\theta(I_a)}
\leq\max\Big(\sup_{I_a}|F_J|,{1\over 2}\sum_{b=1}^B q_b\|F_{J_b}\|_{\mathcal C_\theta(I)}\Big).
\end{equation}
By~\eqref{e:ideas-bound-1} we have
\begin{equation}
  \label{e:panda-2}
\sup_{I_a}|F_J|^2\leq \|F_J\|^2_{\mathcal C_\theta(I_a)}
\leq R.
\end{equation}
Therefore by~\eqref{e:panda-1} and the second inequality in~\eqref{e:ideas-bound-1}
$$
\|F_J\|_{\mathcal C_\theta(I_a)}^2
\leq {1\over 2}\big(R+\sup_{I_a}|F_J|^2\big)
$$
Summing this inequality over $a$ with weights $p_a$,
we see that \eqref{e:contrastep-2} implies
$$
\sum_{a=1}^A p_a\sup_{I_a}|F_J|^2>(1-2\varepsilon_1)R
$$
which gives \eqref{e:contrastep-3}.
\end{proof}
Now, choose $x_a$ as in Lemma~\ref{l:get-rid-of-ders} and put
$$
F_{ab}:=F_{J_b}(x_a)\in\mathbb C,\quad
\omega_{ab}:=\Psi_b(x_a)\in\mathbb R.
$$
Note that by~\eqref{e:F-J-iterative}
$$
F_J(x_a)=\sum_{b=1}^B q_b\exp(i\omega_{ab})F_{ab}.
$$
Using~\eqref{e:basic-inner-product} for $f_b=\exp(i\omega_{ab})F_{ab}$
and~\eqref{e:contrastep-3}, we obtain
\begin{equation}
  \label{e:contrastep-4}
\sum_{a,b} p_aq_b  |F_{ab}|^2
 > (1-2\varepsilon_1)R+
 \sum_{a,b,b'\atop
 b<b'}p_aq_bq_{b'}
 \big|\exp\big(i(\omega_{ab}-\omega_{ab'})\big)F_{ab}-F_{ab'}|^2.
\end{equation}
From the definition~\eqref{e:R-def} of $R$, we have for all $a$
\begin{equation}
  \label{e:F-convex-up}
\sum_{b=1}^B q_b|F_{ab}|^2\leq R.
\end{equation}
Therefore, the left-hand side of~\eqref{e:contrastep-4}
is bounded above by $R$. Using~\eqref{e:pq-bound}, we then get
for all $a,b,b'$ the following approximate
equality featuring the phase terms $\omega_{ab}$:
\begin{equation}
  \label{e:amplitudes-are-near}
\big|\exp\big(i(\omega_{ab}-\omega_{ab'})\big)F_{ab}-F_{ab'}\big|
<\sqrt{2\varepsilon_1 R\over p_{\min}q_{\min}^2}
\leq 2(C'_R)^2L^{\delta+\delta'}
\sqrt{\varepsilon_1 R}.
\end{equation}
Using the smallness of $\varepsilon_1$, we obtain from here a lower bound on $|F_{ab}|$:
\begin{lemm}
For all $a,b$ we have
\begin{equation}
  \label{e:F-lower}
|F_{ab}|\geq {\sqrt{R}\over 2}.
\end{equation}
\end{lemm}
\begin{proof}
By~\eqref{e:epsilon-1} and~\eqref{e:pq-bound}, we have
$p_{\min}\geq 4\varepsilon_1$.
Applying Lemma~\ref{l:extreme-l1} to $\alpha_a=\sum_b q_b|F_{ab}|^2$
and using~\eqref{e:contrastep-4} and~\eqref{e:F-convex-up},
we obtain for all $a$
\begin{equation}
  \label{e:F-lower-int-1}
\sum_{b=1}^B q_b|F_{ab}|^2\geq {R\over 2}.
\end{equation}
We now argue similarly to the proof of~\eqref{e:extreme-l2}.
Fix $a$ and let $F_{a,\min}=\min_b|F_{ab}|$, $F_{a,\max}=\max_b |F_{ab}|$.
By~\eqref{e:F-lower-int-1} we have
$F_{a,\max}\geq \sqrt{R/2}$.
On the other hand the difference
$F_{a,\max}-F_{a,\min}$ is bounded above by~\eqref{e:amplitudes-are-near}.
By~\eqref{e:epsilon-1} we then have
$$
F_{a,\min}\geq \sqrt{R\over 2}-2(C'_R)^2L^{\delta+\delta'}
\sqrt{\varepsilon_1 R} \geq{\sqrt{R}\over 2},
$$
finishing the proof.
\end{proof}
We next estimate the discrepancy between the values $F_{ab}$ for fixed~$b$ and different~$a$,
using the fact that we control the norm $\|F_{J_b}\|_{\mathcal C_\theta(I)}$
and thus the derivative of $F_{J_b}$:
\begin{lemm}
\label{l:FAB-upper}
For all $a,a',b$ we have
$$
|F_{ab}-F_{a'b}|\leq {2\sqrt{R}\,|x_a-x_{a'}|\over\theta|I|}
\leq {2\sqrt{R}\over\theta}L^{H(I)}\cdot|x_a-x_{a'}|.
$$
\end{lemm}
\begin{proof}
This follows immediately by combining Lemma~\ref{l:mvt-theta},
Lemma~\ref{l:F-J-b-lower}, and~\eqref{e:regtree-1}.
\end{proof}
Armed with the bounds obtained above, we are now ready to reach a contradiction
and finish the proof of Lemma~\ref{l:main-step}, using the discrepancy of
the phase shifts $\omega_{ab}$ and the lower bound on $|\partial^2_{xy}\Phi|$ from~\eqref{e:C-Phi}.

Using part~3 of Lemma~\ref{l:regtree} and~\eqref{e:L-restriction}, choose
$a,a',b,b'$ such that
$$
\begin{gathered}
{1\over 2}C_R^{-2/\delta}L^{-2/3}\leq L^{H(I)}\cdot|x_a-x_{a'}|\leq 2L^{-2/3},\\
{1\over 2}C_R^{-2/\delta'}L^{-2/3}\leq L^{H(J)}\cdot|y_b-y_{b'}|\leq 2L^{-2/3}.
\end{gathered}
$$
Recall that $x_a\in I_a$ is chosen in Lemma~\ref{l:get-rid-of-ders}
and $y_b:=y_{J_b}$ is the center of $J_b$. By Lemma~\ref{l:mean-value},
we have for some $(\tilde x,\tilde y)\in I\times J$,
$$
\tau:=\omega_{ab}+\omega_{a'b'}-\omega_{a'b}-\omega_{ab'}=
{(x_a-x_{a'})(y_b-y_{b'})\over h}\,\partial^2_{xy}\Phi(\tilde x,\tilde y).
$$
By~\eqref{e:C-Phi} and~\eqref{e:L-restriction} and since $h=L^{-K}$, $H(I)+H(J)=K-1$, we have
$$
{C_R^{-{2\over\delta}-{2\over\delta'}}\over 8}L^{-1/3} \leq |\tau| \leq 8 L^{-1/3}
\leq \pi.
$$
Therefore, by Lemma~\ref{l:silly} the phase factor $e^{i\tau}$ is bounded away from 1,
which combined with~\eqref{e:F-lower} gives a lower bound on the discrepancy:
\begin{equation}
  \label{e:NLIC-gain}
|F_{ab}|\cdot |e^{i\tau}-1|\geq {|\tau|\sqrt{R}\over\pi}\geq
{C_R^{-{2\over \delta}-{2\over\delta'}}\over 8\pi} L^{-1/3}\sqrt{R}.
\end{equation}
On the other hand we can estimate the same discrepancy from above by~\eqref{e:amplitudes-are-near}, Lemma~\ref{l:FAB-upper},
and the triangle inequality:
$$
\begin{aligned}
|F_{ab}|\cdot |e^{i\tau}-1|
=&\ |e^{i(\omega_{ab}-\omega_{ab'})}F_{ab}-
e^{i(\omega_{a'b}-\omega_{a'b'})}F_{ab}|
\\
\leq & \ |e^{i(\omega_{ab}-\omega_{ab'})}F_{ab}-F_{ab'}|
+|F_{ab'}-F_{a'b'}|\\
&+|e^{i(\omega_{a'b}-\omega_{a'b'})}F_{a'b}-F_{a'b'}|
+|F_{ab}-F_{a'b}|\\
<&\ 4(C'_R)^2L^{\delta+\delta'}\sqrt{\varepsilon_1 R}
+8\theta^{-1}L^{-2/3}\sqrt{R}.
\end{aligned}
$$
Comparing this with~\eqref{e:NLIC-gain} and dividing by $\sqrt{R}$,
we obtain
$$
{C_R^{-{2\over \delta}-{2\over\delta'}}\over 8\pi} L^{-1/3}\ <\
4(C'_R)^2L^{\delta+\delta'}\sqrt{\varepsilon_1}
+8\theta^{-1}L^{-2/3}.
$$
This gives a contradiction with the following
consequences of~\eqref{e:L-restriction} and~\eqref{e:epsilon-1}:
$$
8\theta^{-1}L^{-2/3}\leq {C_R^{-{2\over \delta}-{2\over\delta'}}\over 16\pi} L^{-1/3},\quad
4(C'_R)^2L^{\delta+\delta'}\sqrt{\varepsilon_1}\leq
{C_R^{-{2\over \delta}-{2\over\delta'}}\over 16\pi} L^{-1/3}.
$$

\subsection{Proof of Theorem~\texorpdfstring{\ref{t:main}}{1}}
\label{s:reduce-basic}

We now show how to reduce Theorem~\ref{t:main} to Proposition~\ref{l:almost-main}.
The idea is to split $G$ into pieces using a partition of unity.
On each piece by appropriate rescaling we keep the regularity constant $C_R$ and reduce to the case~\eqref{e:C-Phi} and $h=L^{-K}$ for some fixed $L$ satisfying \eqref{e:L-restriction} and some integer $K>0$.

To be more precise, let $(X,\mu_X)$, $(Y,\mu_Y)$, $\delta,\delta'$, $I_0,J_0$, $\Phi$, $G$
satisfy the hypotheses of Theorem~\ref{t:main}. Using a partition of unity, we write
$G$ as a finite sum
\begin{equation}
  \label{e:G-ell}
G=\sum_\ell G_\ell,\quad
G_\ell\in C^1(I_0\times J_0;\mathbb C),\quad
\supp G_\ell\subset I_\ell\times J_\ell
\end{equation}
where $I_\ell\subset I_0$, $J_\ell\subset J_0$ are intervals such that for some
$m=m(\ell)\in\mathbb Z$,
$$
2^{m-1}<|\partial_{xy}^2\Phi|<2^{m+1}\quad\text{on }I_\ell\times J_\ell.
$$
It then suffices to show~\eqref{e:main} where $G$ is replaced by one of the functions $G_\ell$.
By changing $\Phi$ outside of the support of $G$ (which does not change
the operator $\mathcal B_h$), we then reduce to the case when
\begin{equation}
  \label{e:Phi-normalized}
2^{m-1}<|\partial_{xy}^2\Phi|<2^{m+1}\quad\text{on }I_0\times J_0
\end{equation}
for some $m\in\mathbb Z$. 

We next rescale $\mathcal{B}_{h}$ to an operator $\widetilde{\mathcal{B}}_{\tilde{h}}$ satisfying the hypotheses of Proposition \ref{l:almost-main}. Fix the smallest $L\in\mathbb Z$ satisfying \eqref{e:L-restriction}.
Choose $K\in\mathbb Z$ and $\sigma\in [1,\sqrt{L})$ such that
\begin{equation}
  \label{e:rescaled-param}
\sigma^2=2^m{\tilde h\over h},\quad
\tilde h:=L^{-K}.
\end{equation}
Put for all intervals $I,J$
$$
\begin{aligned}
\widetilde{X}:=\sigma X\subset \tilde{I}_0:=\sigma I_0,&\quad
\widetilde{Y}:=\sigma Y\subset\tilde{J}_0:=\sigma J_0,\\
\mu_{\widetilde{X}}(\sigma I):=\sigma^\delta\mu_X(I),&\quad
\mu_{\widetilde{Y}}(\sigma J):=\sigma^{\delta'}\mu_Y(J).
\end{aligned}
$$
Then $(\widetilde{X},\mu_{\widetilde{X}})$ is $\delta$-regular, and $(\widetilde{Y},\mu_{\widetilde{Y}})$ is $\delta'$-regular,
up to scale $\sigma h$ with regularity constant $C_R$. Consider the unitary operators
$$
\begin{aligned}
U_X:L^2(X,\mu_X)\to L^2(\widetilde X,\mu_{\widetilde X}),&\quad
U_Y:L^2(Y,\mu_Y)\to L^2(\widetilde Y,\mu_{\widetilde Y}),\\
U_Xf(\tilde x)=\sigma^{-\delta/2}f(\sigma^{-1}\tilde x),&\quad
U_Yf(\tilde y)=\sigma^{-\delta'/2}f(\sigma^{-1}\tilde y).
\end{aligned}
$$
Then the operator $\widetilde{\mathcal B}_{\tilde h}:=U_X\mathcal B_h U_Y^{-1}:L^2(\widetilde Y,\mu_{\widetilde Y})
\to L^2(\widetilde X,\mu_{\widetilde X})$ has the form~\eqref{e:B-h}:
$$
\widetilde{\mathcal{B}}_{\tilde{h}}f(\tilde{x})=\int_{\widetilde{Y}}\exp\Big(\frac{i\widetilde{\Phi}(\tilde{x},\tilde{y})}{\tilde{h}}\Big)\widetilde{G}(\tilde{x},\tilde{y})f(\tilde{y})\,d\mu_{\widetilde{Y}}(\tilde{y})
$$
where 
$$
\widetilde{\Phi}(\tilde{x},\tilde{y})=2^{-m}\sigma^2\Phi(\sigma^{-1}\tilde x,\sigma^{-1}\tilde y),\quad
\widetilde G(\tilde{x},\tilde{y})=\sigma^{-\frac{\delta}{2}-\frac{\delta'}{2}}G(\sigma^{-1}\tilde x,\sigma^{-1}\tilde y).
$$
By~\eqref{e:Phi-normalized} the function~$\widetilde\Phi$ satisfies~\eqref{e:C-Phi}.
Fix smallest $K_0\in\mathbb N_0$ such that $\sigma h\leq L^{K_0-K}$, that is
$$
L^{K_0}\geq {2^m\over\sigma}.
$$
Without loss of generality, we may assume that $h$ is small enough depending
on~$L,m$ so that $K\geq 2K_0$. Then Proposition~\ref{l:almost-main} applies to $\widetilde {\mathcal B}_{\tilde h}$ and gives
$$
\|\mathcal B_h\|_{L^2(Y,\mu_Y)\to L^2(X,\mu_X)}=
\|\widetilde {\mathcal B}_{\tilde h}\|_{L^2(\widetilde Y,\mu_{\widetilde Y})\to L^2(\widetilde X,\mu_{\widetilde X})}
\leq C\tilde h^{\varepsilon_0}
\leq C(2^{-m}L)^{\varepsilon_0}h^{\varepsilon_0}
$$
for $\varepsilon_0$ defined in~\eqref{e:epsilon-0} and some constant $C$
depending only on $\delta,\delta',C_R,I_0,J_0,\Phi,G$.
This finishes the proof of Theorem~\ref{t:main}.

\section{Application: spectral gap for hyperbolic surfaces}
  \label{s:hyper}

We now discuss applications of Theorem~\ref{t:main} to spectral gaps.
We start with the case of hyperbolic surfaces, referring the reader to
the book of Borthwick~\cite{BorthwickBook} and to~\cite{hgap}
for the terminology used here.

Let $M=\Gamma\backslash\mathbb H^2$ be a convex co-compact hyperbolic surface,
$\Lambda_\Gamma\subset\mathbb S^1$ be its limit set, $\delta\in [0,1)$ be the dimension
of $\Lambda_\Gamma$, and $\mu$ be the Patterson--Sullivan measure, which is a
probability measure supported on $\Lambda_\Gamma$, see for instance~\cite[\S14.1]{BorthwickBook}.
Since $\Lambda_\Gamma$ is closed and is not equal to the entire $\mathbb S^1$,
we may cut the circle $\mathbb S^1$ to turn it into an interval and treat
$\Lambda_\Gamma$ as a compact subset of $\mathbb R$. Then $(\Lambda_\Gamma,\mu)$
is $\delta$-regular up to scale~0 with some constant $C_R$, see for instance~\cite[Lemma~14.13]{BorthwickBook}.
The regularity constant $C_R$ depends continuously on the surface, as explained
in the case of three-funnel surfaces in~\cite[Proposition~7.7]{hgap}.

The main result of this section is the following essential spectral gap for $M$.
We formulate it here in terms of scattering resolvent of the Laplacian.
Another formulation is in terms of a zero free region for the Selberg
zeta function past the first pole, see for instance~\cite{hgap}.
See below for a discussion of previous work on spectral gaps.
\begin{theo}
  \label{t:gap-hyper}
Consider the meromorphic scattering resolvent
$$
R(\lambda)=\Big(-\Delta_M-{1\over 4}-\lambda^2\Big)^{-1}:\begin{cases}
L^2(M)\to L^2(M),\quad \Im\lambda >0,\\
L^2_{\comp}(M)\to L^2_{\loc}(M),\quad \Im\lambda\leq 0.
\end{cases}
$$
Assume that $0<\delta<1$. Then $M$ has an essential spectral gap of size
\begin{equation}
  \label{e:gap-hyper}
\beta={1\over 2}-\delta+(13C_R)^{-{320\over\delta(1-\delta)}}
\end{equation}
that is $R(\lambda)$ has only finitely many poles
in $\{\Im\lambda>-\beta\}$ and it satisfies the cutoff estimates
for each $\psi\in C_0^\infty(M),\varepsilon>0$
and some constant $C_0$ depending on $\varepsilon$
$$
\|\psi R(\lambda)\psi\|_{L^2\to L^2}\leq C(\psi,\varepsilon)|\lambda|^{-1-2\min(0,\Im\lambda)+\varepsilon},\quad
\Im\lambda \in [-\beta,1],\quad
|\Re\lambda|\geq C_0.
$$
\end{theo}
\begin{proof}
We use the strategy of~\cite{hgap}. By~\cite[Theorem~3]{hgap},
it suffices to show the following fractal uncertainty principle:
for each $\rho\in (0,1)$,
$$
\beta_0:={1\over 2}-\delta+(150C_R^2)^{-{160\over\delta(1-\delta)}},
$$
and each cutoff function $\chi\in C^\infty(\mathbb S^1\times \mathbb S^1)$
supported away from the diagonal, there exists a constant
$C$ depending on $M,\chi,\rho$ such that for all $h\in (0,1)$
\begin{equation}
  \label{e:fup-hyper}
\|\mathcal \indic_{\Lambda_\Gamma(h^\rho)}B_{\chi,h}\indic_{\Lambda_\Gamma(h^\rho)}\|_{L^2(\mathbb S^1)\to L^2(\mathbb S^1)}
\leq Ch^{\beta_0-2(1-\rho)}
\end{equation}
where $\Lambda_\Gamma(h^\rho)\subset \mathbb S^1$ is the $h^\rho$ neighborhood of $\Lambda_\Gamma$
and the operator $B_{\chi,h}$ is defined by
(here $|x-y|$ is the Euclidean distance between $x,y\in\mathbb S^1\subset\mathbb R^2$)
$$
B_{\chi,h}f(x)=(2\pi h)^{-1/2}\int_{\mathbb S^1}|x-y|^{2i/h}\chi(x,y) f(y)\,dy.
$$
To show~\eqref{e:fup-hyper}, we first note that by Lemma~\ref{l:fatten-lebesgue},
$(Y,\mu_Y)$ is $\delta$-regular up to scale~$h$ with constant $30C_R^2$, where 
$Y=\Lambda_\Gamma(h)$ and $\mu_Y$ is $h^{\delta-1}$ times the restriction
of the Lebesgue measure to $Y$. We lift $\chi(x,y)$ to a compactly supported
function on $\mathbb R^2$ (splitting it into pieces using a partition of unity) and write 
$$
B_{\chi,h} \indic_{\Lambda_\Gamma(h)} f(x)=(2\pi)^{-1/2}h^{1/2-\delta}\mathcal B_h f(x),
$$
where $\mathcal B_h$ has the form~\eqref{e:B-h}
with $G(x,y)=\chi(x,y)$ and (with $|x-y|$ still denoting the Euclidean
distance between $x,y\in\mathbb S^1$)
$$
\Phi(x,y)=2\log|x-y|.
$$
The function $\Phi$ is smooth and satisfies the condition
$\partial^2_{xy}\Phi\neq 0$ on the open set $\mathbb S^1\times\mathbb S^1\setminus \{x=y\}$
which contains the support of $G$, see for instance~\cite[\S4.3]{fullgap}. Applying
Theorem~1 with $(X,\mu_X):=(Y,\mu_Y)$, we obtain
$$
\|\indic_{\Lambda_\Gamma(h)}B_{\chi,h}\indic_{\Lambda_\Gamma(h)}\|_{L^2(\mathbb S^1)\to L^2(\mathbb S^1)}\leq
Ch^{\beta_0}.
$$
Similarly we have
$$
\|\indic_{\Lambda_\Gamma(h)+t}B_{\chi,h}\indic_{\Lambda_\Gamma(h)+s}\|_{L^2(\mathbb S^1)\to L^2(\mathbb S^1)}\leq
Ch^{\beta_0},\quad
t,s\in [-1,1]
$$
where $X+t$ is the result of rotating $X\subset\mathbb S^1$ by angle $t$.
Covering $\Lambda_\Gamma(h^\rho)$ with at most $10h^{\rho-1}$ rotations
of the set $\Lambda_\Gamma(h)$ (see for instance the proof of~\cite[Proposition~4.2]{fullgap})
and using triangle inequality, we obtain~\eqref{e:fup-hyper}, finishing the proof.
\end{proof}
We now briefly discuss previous results on spectral gaps for hyperbolic surfaces:
\begin{itemize}
\item The works of Patterson~\cite{Patterson} and Sullivan~\cite{Sullivan} imply that $R(\lambda)$
has no poles with $\Im\lambda>\delta-{1\over 2}$. On the other hand, the fact
that $R(\lambda)$ is the $L^2$ resolvent of the Laplacian in $\{\Im\lambda>0\}$
shows that it has only has finitely many poles in this region.
Together these two results give the essential spectral gap
$\beta=\max(0,{1\over 2}-\delta)$. Thus Theorem~\ref{t:gap-hyper} gives no new results
when $\delta$ is much larger than~$1\over 2$.
\item Using the method developed by Dolgopyat~\cite{dolgop}, Naud~\cite{NaudGap}
showed an essential spectral gap of size $\beta>{1\over 2}-\delta$
when $\delta>0$. Oh--Winter~\cite{OhWinter} showed that the size of the
gap is uniformly controlled for towers of congruence covers in the arithmetic case.
\item Dyatlov--Zahl~\cite{hgap} introduced the fractal uncertainty principle
approach to spectral gaps and used it together with tools
from additive combinatorics to give an estimate of the size of the gap in terms
of $C_R$ in the case when $\delta$ is very close to $1\over 2$.
\item Bourgain--Dyatlov~\cite{fullgap} showed that each convex co-compact
hyperbolic surface has an essential spectral gap of some size $\beta=\beta(\delta,C_R)>0$.
Their result is new in the case $\delta>{1\over 2}$ and is thus complementary
to the results mentioned above as well as to Theorem~\ref{t:gap-hyper}.
\end{itemize}
More generally, spectral gaps have been studied for noncompact manifolds
with hyperbolic trapped sets.
(See for instance~\cite[\S2.1]{Nonnenmacher} for a definition.)
In this setting the Patterson--Sullivan gap ${1\over 2}-\delta$
generalizes to the \emph{pressure gap} $-P({1\over 2})$
which has been established by Ikawa~\cite{Ikawa}, Gaspard--Rice~\cite{GaspardRice},
and Nonnenmacher--Zworski~\cite{NonnenmacherZworskiActa}.
An improved gap $\beta> -P({1\over 2})$ has been proved in several cases,
see in particular Petkov--Stoyanov~\cite{PetkovStoyanov}
and Stoyanov~\cite{Stoyanov1,Stoyanov2}.
We refer the reader to the review of Nonnenmacher~\cite{Nonnenmacher} for an overview
of results on spectral gaps for general hyperbolic trapped sets.

\section{Application: spectral gap for open quantum maps}
\label{s:oqm}

In this section, we discuss applications of fractal uncertainty principle
to the spectral properties of open quantum maps.
Following the notation in~\cite{oqm} we consider an open quantum baker's map $B_N$
determined by a triple $(M,\mathcal{A},\chi)$ where $M\in\mathbb{N}$ is called the base, $\mathcal{A}\subset\mathbb{Z}_M=\{0,1,\ldots,M-1\}$ is called the alphabet, and $\chi\in C_0^\infty((0,1);[0,1])$ is a cutoff function.
The map $B_N$ is a sequence of operators 
 $B_N:\ell^2_N\to\ell^2_N$,
 $\ell^2_N=\ell^2(\mathbb Z_N)$, defined for every positive $N\in M\mathbb Z$ by
\begin{equation}
  \label{e:oqm-def}
B_N=\mathcal F_N^*\begin{pmatrix}
\chi_{N/M}\mathcal F_{N/M}\chi_{N/M} & \\
&\ddots&\\
 & &\chi_{N/M}\mathcal F_{N/M}\chi_{N/M}
\end{pmatrix}I_{\mathcal A, M}
\end{equation}
where $\mathcal F_N$ is the unitary Fourier transform given by the
$N\times N$ matrix $\frac{1}{\sqrt{N}}(e^{-2\pi ij\ell/N})_{j\ell}$, $\chi_{N/M}$ is the multiplication operator
on~$\ell^2_{N/M}$ discretizing $\chi$, and
$I_{\mathcal A,M}$ is the diagonal matrix
with $\ell$-th diagonal entry equal to~1 if $\lfloor {\ell\over N/M}\rfloor\in\mathcal A$
and 0 otherwise. 

An important difference from~\cite{oqm} is that in the present paper we allow $N$ to be any multiple of $M$,
while~\cite{oqm} required that $N$ be a power of $M$. To measure the size of $N$, we let $k$ be the unique integer such that $M^k\leq N<M^{k+1}$, i.e. $k=\lfloor\frac{\log N}{\log M}\rfloor$. Denote
by $\delta$ the dimension of the Cantor set corresponding to $M$ and $\mathcal A$, given by
$$
\delta=\frac{\log|\mathcal{A}|}{\log M}.
$$
The main result of this section is the following spectral gap, which was previously
established in~\cite[Theorem~1]{oqm} for the case when $N$ is a power of $M$:
\begin{theo}
  \label{t:gap-improves}
Assume that $0<\delta<1$, that is $1<|\mathcal A|<M$. Then
there exists
\begin{equation}
  \label{e:improved-beta}
\beta=\beta(M,\mathcal A)>\max\Big(0,{1\over 2}-\delta\Big)
\end{equation}
such that, with $\Sp(B_N)\subset \{\lambda\in \mathbb C\colon |\lambda|\leq 1\}$ denoting
the spectrum of $B_N$,
\begin{equation}
  \label{e:gap-improves}
\limsup_{N\to\infty,\,N\in M\mathbb Z}\max\{|\lambda|\colon \lambda\in\Sp(B_N)\}\ \leq\ M^{-\beta}.
\end{equation}
\end{theo}
The main component of the proof is a fractal uncertainty principle.
For the case $N=M^k$, the following version of it was used in~\cite{oqm}:
\begin{equation}
  \label{e:fup-power}
\|\indic_{\mathcal{C}_k}\mathcal{F}_N\indic_{\mathcal{C}_k}\|_{\ell^2_N\to\ell^2_N}\leq CN^{-\beta}
\end{equation}
where $\mathcal{C}_k$ is the discrete Cantor set given by
\begin{equation}
  \label{e:C-k}
\mathcal C_k:=\Big\{\sum_{j=0}^{k-1} a_j M^j\,\Big|\, a_0,\dots,a_{k-1}\in\mathcal A\Big\}\subset \mathbb Z_N.
\end{equation}
For general $N\in M\mathbb Z\cap [M^k,M^{k+1})$, we define a similar discrete Cantor set in $\mathbb{Z}_N$ by
\begin{equation}
  \label{e:C-k-N}
\mathcal C_k(N):=\big\{b_j(N)\colon j\in\mathcal C_k\big\}\subset\mathbb Z_N,\quad
b_j(N):=\left\lceil \frac{jN}{M^k}\right\rceil.
\end{equation}
In fact, in our argument
we only need $b_j(N)$ to be some integer in $[\frac{j N}{M^k},\frac{(j+1)N}{M^k})$. 

The uncertainty principle then takes the following form:
\begin{theo}
\label{t:fup-oqm}
Assume that $0<\delta<1$. Then there exists
\begin{equation}
\beta=\beta(M,\mathcal A)>\max\Big(0,{1\over 2}-\delta\Big)
\end{equation}
such that for some constant $C$ and all $N$,
\begin{equation}
\label{e:fup-oqm}
\|\indic_{\mathcal C_k(N)}\mathcal{F}_N\indic_{\mathcal C_k(N)}\|_{\ell^2_N\to\ell^2_N}\leq CN^{-\beta}.
\end{equation}
\end{theo}
In~\S\ref{s:oqm-gap} below, we show that Theorem~\ref{t:fup-oqm} implies Theorem~\ref{t:gap-improves}.
We prove Theorem~\ref{t:fup-oqm} in~\S\S\ref{s:oqm-fup-1},\ref{s:oqm-fup-2}, using
Ahlfors--David regularity of the Cantor set which is verified in~\S\ref{s:oqm-regular}.

\subsection{Fractal uncertainty principle implies spectral gap}
\label{s:oqm-gap}

We first show that Theorem \ref{t:fup-oqm} implies Theorem \ref{t:gap-improves}. The argument is essentially the same as in \cite[Section 2.3]{oqm}, relying on the following generalization of~\cite[Proposition~2.5]{oqm}:
\begin{prop}[Localization of eigenstates]
  \label{l:key-eigenvalues}
Fix $\nu>0$, $\rho\in (0,1)$, and assume that for some $k\in \mathbb N$,
$N\in M\mathbb Z\cap [M^k,M^{k+1})$,
$\lambda\in\mathbb C$, $u\in\ell^2_N$, we have
$$
B_N u=\lambda u,\quad |\lambda|\geq M^{-\nu}.
$$
Define
$$
X_\rho:=\bigcup\{\mathcal{C}_k(N)+m: m\in\mathbb{Z},\ |m|\leq(M+2)N^{1-\rho}\}\subset\mathbb Z_N.
$$
Then
\begin{gather}
  \label{e:concentration-3}
\|u\|_{\ell^2_N}\leq M^\nu|\lambda|^{-\rho k}\,\|\indic_{X_{\rho}}u\|_{\ell^2_N}+\mathcal O(N^{-\infty})\|u\|_{\ell^2_N},\\
  \label{e:concentration-4}
\|u-\mathcal F_N^*\indic_{X_\rho}\mathcal F_N u\|_{\ell^2_N}=\mathcal O(N^{-\infty})\|u\|_{\ell^2_N}
\end{gather}
where the constants in $\mathcal O(N^{-\infty})$ depend only on $\nu,\rho,\chi$.
\end{prop}
\begin{proof}
Following~\cite[(2.7)]{oqm}, let
$\Phi=\Phi_{M,\mathcal{A}}$ be the expanding map defined by
\begin{equation}
  \label{e:Phi}
\Phi:\bigsqcup_{a\in\mathcal A}\left(\frac{a}{M},\frac{a+1}{M}\right)\to (0,1);\quad
\Phi(x)=Mx-a,\quad
x\in \left(\frac{a}{M},\frac{a+1}{M}\right).
\end{equation}
Put
\begin{equation}
  \label{e:tilde-k}
\tilde k:=\lceil\rho k\rceil\in \{1,\dots,k\}.
\end{equation}
With $d(\cdot,\cdot)$ denoting the distance function on the circle as in~\cite[\S2.1]{oqm},
define
$$
\mathcal X_{\rho}:=\{x\in [0,1]\colon d(x,\Phi^{-\tilde k}([0,1]))\leq N^{-\rho}\}.
$$
Then~\eqref{e:concentration-3}, \eqref{e:concentration-4} follow from the long
time Egorov theorem~\cite[Proposition~2.4]{oqm} (whose proof never used
that $N$ is a power of $M$) similarly to~\cite[Proposition~2.5]{oqm},
as long as we show the following analog of~\cite[(2.30)]{oqm}:
\begin{equation}
  \label{e:cover-x-rho}
\ell\in \{0,\dots,N-1\},\quad {\ell\over N}\in \mathcal X_\rho\quad\Longrightarrow\quad \ell\in X_\rho.
\end{equation}
To see~\eqref{e:cover-x-rho}, note that (with the intervals considered in $\mathbb R/\mathbb Z$)
$$
\Phi^{-\tilde k}([0,1])\subset\bigcup_{j\in\mathcal C_k}\Big(
{j-M^{k-\tilde k}\over M^k},{j+M^{k-\tilde k}\over M^k}\Big).
$$
Assume that $\ell\in \{0,\dots,N-1\}$ and
$\ell/N\in\mathcal X_\rho$. Then there exists $j\in\mathcal C_k$ such that
$$
d\Big({\ell\over N},{j\over M^k}\Big)\leq N^{-\rho}+M^{-\tilde k}
\leq (M+1)N^{-\rho}.
$$
It follows that
$$
d\Big({\ell\over N},{b_j(N)\over N}\Big)\leq (M+2)N^{-\rho}
$$
and thus $\ell\in X_\rho$ as required.
\end{proof}
Now, we assume that Theorem~\ref{t:fup-oqm} holds
and prove Theorem~\ref{t:gap-improves}.
Using the triangle inequality as in the proof of~\cite[Proposition~2.6]{oqm},
we obtain
\begin{equation}
  \label{e:oqm-gap-almost}
\begin{aligned}
\|\indic_{X_\rho}\mathcal F_N^*\indic_{X_\rho}\|_{\ell^2_N\to\ell^2_N}
&\leq (2M+5)^2N^{2(1-\rho)}\|\indic_{\mathcal{C}_k(N)}\mathcal F_N \indic_{\mathcal{C}_k(N)}\|_{\ell^2_N\to\ell^2_N}\\
&\leq CN^{2(1-\rho)-\beta}.
\end{aligned}
\end{equation}
Here $C$ denotes a constant independent of $N$.

Assume that $\lambda\in \mathbb C$ is an eigenvalue
of $B_N$ such that $|\lambda|\geq M^{-\beta}$ and $u\in\ell^2_N$ is
a normalized eigenfunction of $B_N$ with eigenvalue $\lambda$.
By~\eqref{e:concentration-3}, \eqref{e:concentration-4},
and~\eqref{e:oqm-gap-almost}
\begin{equation}
  \label{e:fupuse}
\begin{aligned}
1=\|u\|_{\ell^2_N}&\leq M^\beta|\lambda|^{-\rho k}\|\indic_{X_\rho} u\|_{\ell^2_N}+\mathcal O(N^{-\infty})\\
&\leq M^\beta|\lambda|^{-\rho k}\|\indic_{X_\rho}\mathcal F_N^*\indic_{X_\rho}\mathcal F_N u\|_{\ell^2_N}+\mathcal O(N^{-\infty})\\
&\leq C|\lambda|^{-\rho k}N^{2(1-\rho)-\beta}
+\mathcal O(N^{-\infty}).
\end{aligned}
\end{equation}
It follows that $|\lambda|^{\rho k}\leq CN^{-\beta+2(1-\rho)}$
or equivalently
$$
|\lambda|\leq C^{1/\rho k}M^{(2(1-\rho)-\beta)/\rho}.
$$
This implies that
$$
\limsup_{N\to\infty}\max\{|\lambda|\colon \lambda\in\Sp(B_N)\}\ \leq\max\{M^{-\beta},M^{(2(1-\rho)-\beta)/\rho}\}.
$$
Letting $\rho\to1$, we conclude the proof
of Theorem~\ref{t:gap-improves}.

\subsection{Regularity of discrete Cantor sets}
\label{s:oqm-regular}

Theorem~\ref{t:fup-oqm} will be deduced from Theorem~\ref{t:main}
and the results of~\cite{fullgap}. To apply these, we establish Ahlfors--David
regularity of the Cantor set $\mathcal C_k(N)\subset\mathbb{Z}_N=\{0,\dots,N-1\}$ in the following
discrete sense.
\begin{defi}
  \label{d:regular-discrete}
We say that $X\subset \mathbb{Z}_N$ is $\delta$-regular
with constant~$C_R$ if
\begin{itemize}
\item for each interval $J$ of size $|J|\geq1$, we have
$\#(J\cap X)\leq C_R |J|^\delta$, and
\item for each interval $J$ with $1\leq |J|\leq N$ which is centered at a point in $X$, we have $\#(J\cap X)\geq C_R^{-1}|J|^\delta$.
\end{itemize}
\end{defi}
Definition~\ref{d:regular-discrete} is related
to Definition~\ref{d:regular-set} as follows:
\begin{lemm}
  \label{l:rel-reg}
Let $X\subset\mathbb{Z}_N$. Define $\widetilde X:=N^{-1}X\subset [0,1]$
which supports the measure
\begin{equation}
  \label{e:mu-discrete}
\mu_{\widetilde X}(A):=N^{-\delta}\cdot \#(\widetilde X\cap A),\quad
A\subset\mathbb R.
\end{equation}
Then $X$
is $\delta$-regular with constant $C_R$ 
in the sense of Definition~\ref{d:regular-discrete}
if and only if $(\widetilde X,\mu_{\widetilde X})$
is $\delta$-regular up to scale $N^{-1}$ with constant $C_R$
in the sense of Definition~\ref{d:regular-set}.
\end{lemm}
\begin{proof}
This follows directly from the two definitions.
\end{proof}
We first establish the regularity
of the discrete Cantor set $\mathcal C_k$ defined in~\eqref{e:C-k}:
\begin{lemm}
  \label{l:reg-C-k}
The set $\mathcal{C}_k\subset\mathbb{Z}_{M^k}$ is $\delta$-regular with constant $C_R=2M^{2\delta}$.
\end{lemm}
\begin{proof}
We notice that for all integers $k'\in[0,k]$ and $j'\in\mathbb Z$
\begin{equation}
  \label{e:basic-count}
\#\big(\mathcal C_k\cap[j'M^{k'},(j'+1)M^{k'})\big)=
\begin{cases}
|\mathcal{A}|^{k'}=M^{\delta k'}
,&j'\in\mathcal{C}_{k-k'};\\ 
0,
&j'\not\in\mathcal{C}_{k-k'}.
\end{cases}
\end{equation}
Let $J$ be an interval in $\mathbb{R}$, with $1\leq|J|\leq N=M^k$. Choose an integer
$k'\in[0,k-1]$ such that $M^{k'}\leq|J|\leq M^{k'+1}$. 
Then there exists some $j'\in\mathbb Z$ such that 
$$
J\subset[j'M^{k'+1},(j'+2)M^{k'+1}).
$$
Therefore by~\eqref{e:basic-count}
$$
\#(\mathcal C_k\cap J)\leq 2M^{\delta(k'+1)}\leq 2M^\delta|J|^\delta\leq C_R|J|^\delta.
$$
On the other hand, if $|J|>N$ then
$$
\#(\mathcal C_k\cap J)\leq \#(\mathcal C_k)=N^\delta\leq |J|^\delta.
$$
This gives the required upper bound on $\#(\mathcal C_k\cap J)$.

Now, assume that $1\leq |J|\leq N$ and $J$ is centered at some $j\in\mathcal{C}_k$.
Choose $k'$ as before. If $k'=0$ then
$$
\#(\mathcal C_k\cap J)\geq 1\geq M^{-\delta}|J|^\delta\geq C_R^{-1}|J|^\delta.
$$
We henceforth assume that $1\leq k'\leq k-1$.
Let $j'\in\mathcal{C}_{k-k'+1}$ be the unique element such that 
$j'M^{k'-1}\leq j<(j'+1)M^{k'-1}$.
Since $M\geq 2$, we have $|J|\geq M^{k'}\geq 2M^{k'-1}$ and thus
$$
[j'M^{k'-1},(j'+1)M^{k'-1}]\subset[j-M^{k'-1},j+M^{k'-1}]\subset J.
$$
Therefore by~\eqref{e:basic-count}
$$
\#(\mathcal C_k\cap J)\geq M^{\delta(k'-1)}\geq M^{-2\delta}|J|^\delta\geq C_R^{-1}|J|^\delta.
$$
This gives the required lower bound on $\#(\mathcal C_k\cap J)$, finishing the proof.
\end{proof}
We now establish regularity of the dilated Cantor set $\mathcal C_k(N)$:
\begin{prop}
  \label{l:reg-C-k-N}
Assume that $M^k\leq N<M^{k+1}$ and let $\mathcal{C}_k(N)\subset\mathbb{Z}_N$ 
be given by~\eqref{e:C-k-N}. Then $\mathcal C_k(N)$
is $\delta$-regular with constant $C_R=8M^{3\delta}$.
\end{prop}
\begin{proof}
For any interval $J$, we have
$$
\#(\mathcal{C}_k(N)\cap J)=\#\{j\in\mathcal{C}_k\colon b_j(N)\in J\}
=\#\Big\{j\in\mathcal C_k\colon{M^k\over N}b_j(N)\in {M^k\over N}J\Big\}.
$$
By our choice of $b_j(N)$, we have ${M^k\over N}b_j(N) \in [j,j+1)$. Therefore
$$
\#\Big(\mathcal{C}_k\cap {M^k\over N}J\Big)-1\leq\#(\mathcal{C}_k(N)\cap J)
\leq\#\Big(\mathcal{C}_k\cap {M^k\over N}J\Big)+1.
$$
We apply Lemma~\ref{l:reg-C-k} to see that for any interval $J$ with $|J|\geq 1$
$$
\#(\mathcal{C}_k(N)\cap J)
\leq 2M^{2\delta}|J|^\delta
+1
\leq 3M^{2\delta}|J|^\delta
\leq C_R|J|^\delta.
$$
Now, assume that $J$ is an interval with $8^{1/\delta}M^3\leq |J|\leq N$ centered at
$b_j(N)$ for some $j\in\mathcal C_k$. Then ${M^k\over N}J$ contains the interval
of size ${1\over 2M}|J|$ centered at $j$.
Therefore, by Lemma~\ref{l:reg-C-k}
$$
\#(\mathcal{C}_k(N)\cap J)
\geq
{1\over 2M^{2\delta}}\Big({|J|\over 2M}\Big)^\delta-1
\geq {|J|^\delta\over 8M^{3\delta}}\geq C_R^{-1}|J|^\delta.
$$
Finally, if $J$ is an interval with $1\leq |J|\leq 8^{1/\delta}M^3$ centered
at a point in $\mathcal C_k(N)$, then
$$
\#(\mathcal C_k(N)\cap J)\geq 1\geq C_R^{-1}|J|^\delta.\qedhere
$$
\end{proof}

\subsection{Fractal uncertainty principle for \texorpdfstring{$\delta\leq 1/2$}{small delta}}
\label{s:oqm-fup-1}

The proof of Theorem \ref{t:fup-oqm} in the case $\delta\leq 1/2$
relies on the following corollary of Theorem~\ref{t:main}:
\begin{prop}
  \label{l:fup-oqm-general}
Let $X,Y\subset\mathbb{Z}_N$ be $\delta$-regular with constant $C_R$
and $0<\delta<1$. Then 
\begin{equation}
  \label{e:fup-oqm-general}
\|\indic_X\mathcal{F}_N\indic_Y\|_{\ell^2_N\to\ell^2_N}
\leq CN^{-(\frac{1}{2}-\delta+\varepsilon_0)}
\end{equation}
where $C$ only depends on $\delta,C_R$ and
\begin{equation}
  \label{e:epsilon-0-oqm}
\varepsilon_0=(5C_R)^{-\frac{160}{\delta(1-\delta)}}.
\end{equation}
\end{prop}
\begin{proof}
Put $h:=N^{-1}$, $\widetilde X:=hX$, $\widetilde Y:=hY$, and define
the measures $\mu_{\widetilde X},\mu_{\widetilde Y}$ by~\eqref{e:mu-discrete}.
By Lemma~\ref{l:rel-reg}, $(\widetilde{X},\mu_{\widetilde{X}})$ and $(\widetilde{Y},\mu_{\widetilde{Y}})$ are $\delta$-regular up to scale $h$ with constant $C_R$.
Consider the operator $\mathcal B_h:L^1(\widetilde Y,\mu_{\widetilde Y})\to L^\infty(\widetilde X,\mu_{\widetilde X})$ defined by
$$
\mathcal{B}_hf(x)=\int_{\widetilde{Y}}\exp\Big(-{2\pi ixy\over h}\Big)f(y)\,d\mu_{\widetilde{Y}}(y)
$$
and note that it has the form~\eqref{e:B-h} with $\Phi(x,y)=-2\pi xy$, $G\equiv 1$.
By Theorem~\ref{t:main}
$$
\|\mathcal{B}_h\|_{L^2(\widetilde{Y},\mu_{\widetilde{Y}})\to L^2(\widetilde{X},\mu_{\widetilde{X}})}\leq Ch^{\varepsilon_0}.
$$
Comparing the formula
$$
\mathcal B_hf\Big({j\over N}\Big)=N^{-\delta} \sum_{\ell\in Y} \exp\Big(-{2\pi ij\ell\over N}\Big)
f\Big({\ell\over N}\Big),\quad j\in X
$$
with the definition of the discrete Fourier transform $\mathcal F_N$, we see that
$$
\|\indic_X \mathcal F_N\indic_Y\|_{\ell^2_N\to\ell^2_N}
=N^{\delta-1/2}\|\mathcal B_h\|_{L^2(\widetilde{Y},\mu_{\widetilde Y})\to
L^2(\widetilde X,\mu_{\widetilde X})}
$$
which finishes the proof.
\end{proof}
Combining Propositions~\ref{l:reg-C-k-N} and~\ref{l:fup-oqm-general}, we get~\eqref{e:fup-oqm} for 
\begin{equation}
\beta =\frac{1}{2}-\delta+(40M^{3\delta})^{-\frac{160}{\delta(1-\delta)}}
\end{equation}
which finishes the proof of Theorem~\ref{t:fup-oqm} for $\delta\leq 1/2$.

\subsection{Fractal uncertainty principle for \texorpdfstring{$\delta> 1/2$}{large delta}}
\label{s:oqm-fup-2}

For $\delta>1/2$, Theorem~\ref{t:main} does not in general 
give an improvement over the trivial gap $\beta=0$.
Instead, we shall use the following reformulation of~\cite[Theorem~4]{fullgap}:
\begin{prop}
  \label{l:full-gap}
Let $0\leq\delta<1$, $C_R\geq1$, $N\geq1$ and assume that
$\widetilde X,\widetilde Y\subset [-1,1]$ and $(\widetilde X,\mu_{\widetilde X})$ and $(\widetilde Y,\mu_{\widetilde Y})$ are $\delta$-regular
up to scale $N^{-1}$ with constant $C_R$
in the sense of Definition~\ref{d:regular-set},
for some
finite measures $\mu_{\widetilde X},\mu_{\widetilde Y}$ supported on $\widetilde X,\widetilde Y$.

Then there exist $\beta_0>0$, $C_0$ depending only on $\delta$, $C_R$ such that for all $f\in L^2(\mathbb{R})$,
\begin{equation}
  \label{e:fup-fourier}
\supp\hat{f}\subset N\cdot \widetilde Y\quad
\Longrightarrow \quad \|f\|_{L^2(\widetilde X)}\leq C_0N^{-\beta_0}\|f\|_{L^2(\mathbb{R})}.
\end{equation}
Here $\hat f$ denotes the Fourier transform of $f$:
\begin{equation}
\hat{f}(\xi)=\mathcal{F}f(\xi)=\int_{\mathbb{R}}e^{-2\pi ix\xi}f(x)dx.
\end{equation}
\end{prop}
Proposition~\ref{l:full-gap} implies the following discrete fractal uncertainty principle:
\begin{prop}
  \label{l:fullgap-oqm}
Let $X,Y\subset\mathbb Z_N$ be $\delta$-regular with constant
$C_R$ and $0\leq \delta<1$. Then
\begin{equation}
  \label{e:fullgap-oqm}
\|\indic_X\mathcal F_N\indic_Y\|_{\ell^2_N\to \ell^2_N}\leq CN^{-\beta}
\end{equation}
where $C,\beta>0$ only depend on $\delta,C_R$.
\end{prop}
\begin{proof}
Put $h:=N^{-1}$,
$$
\widetilde X:=h X+[-h,h],\quad
\widetilde Y:=h Y+[-h,h],
$$
and define the measures $\mu_{\widetilde X},\mu_{\widetilde Y}$
on $\widetilde X,\widetilde Y$
by~\eqref{e:fatten-lebesgue}.
By Lemmas~\ref{l:rel-reg} and~\ref{l:fatten-lebesgue},
$(\widetilde X,\mu_{\widetilde X})$
and $(\widetilde Y,\mu_{\widetilde Y})$ are $\delta$-regular
up to scale $h$ with constant $30C_R^2$.
Applying Proposition~\ref{l:full-gap}, we obtain
for some constants $\beta_0>0,C_0$ depending only on $\delta,C_R$
and all $f\in L^2(\mathbb R)$
\begin{equation}
  \label{e:fullgap-2}
\supp\hat{f}\subset N\cdot \widetilde Y\quad
\Longrightarrow \quad \|f\|_{L^2(\widetilde X)}\leq C_0N^{-\beta_0}\|f\|_{L^2(\mathbb{R})}.
\end{equation}
To pass from~\eqref{e:fullgap-2} to~\eqref{e:fullgap-oqm}, fix
a cutoff function $\chi$ such that for some constant $c>0$
$$
\chi\in C_0^\infty\big((-1/2,1/2)\big),\quad
\|\chi\|_{L^2}=1,\quad
\inf_{[0,1]}|\mathcal F^{-1}\chi|\geq c.
$$
This is possible since for any $\chi\in C_0^\infty(\mathbb R)$
which is not identically 0, $\mathcal F^{-1}\chi$ extends to
an entire function and thus has no zeroes on $\{\Im z=s\}$
for all but countably many choices of $s\in\mathbb R$.
Choosing such $s$ we see that $\mathcal F^{-1}(e^{-s\xi}\chi(\xi))$
has no real zeroes.

Now, take arbitrary $u\in \ell^2_N$. Consider the function $f\in L^2(\mathbb R)$ defined by
$$
\hat f(\xi)=\sum_{\ell\in Y}u(\ell)\chi(\xi-\ell).
$$
Then
$\supp\hat f\subset N\cdot\widetilde Y$
and $\|f\|_{L^2(\mathbb R)}\leq \|u\|_{\ell^2_N}$,
so by~\eqref{e:fullgap-2}
\begin{equation}
  \label{e:fullgap-int-1}
\|f\|_{L^2(\widetilde X)}\leq C_0N^{-\beta_0}\|u\|_{\ell^2_N}.
\end{equation}
On the other hand, for all $j\in \mathbb Z_N$, we have for all $j\in X$
\begin{equation}
  \label{e:fourier-corr}
{1\over \sqrt N}f\Big({j\over N}\Big)=\mathcal F_N^* \indic_Y u(j)\cdot(\mathcal F^{-1}\chi)\Big({j\over N}\Big).
\end{equation}
Consider the nonoverlapping collection of intervals
$$
I_j:=\Big[{j\over N}-{1\over 2N},{j\over N}+{1\over 2N}\Big]\subset\widetilde X,\quad
j\in X.
$$
Using that $(|f|^2)'=2\Re(\bar f f')$, we have
$$
|\mathcal F_N^*\indic_Y u(j)|^2\leq{1\over c^2N}\Big|f\Big({j\over N}\Big)\Big|^2
\leq C\int_{I_j} |f(x)|^2\,dx
+{C\over N}\int_{I_j} |f(x)|\cdot |f'(x)|\,dx
$$
where $C$ denotes some constant depending only on $\delta,C_R,\chi$. 
Summing over $j\in X$ and using the Cauchy--Schwarz inequality, we obtain
$$
\|\indic_X\mathcal F_N^*\indic_Yu\|_{\ell^2_N}^2\leq
C\|f\|_{L^2(\widetilde X)}^2
+{C\over N}\|f\|_{L^2(\widetilde X)}\cdot \|f'\|_{L^2(\mathbb R)}.
$$
Since $\supp \hat f\subset [-N,N]$,
we have $\|f'\|_{L^2(\mathbb R)}\leq 10N\|f\|_{L^2(\mathbb R)}\leq 10N\|u\|_{\ell^2_N}$ and thus
by~\eqref{e:fullgap-int-1}
$$
\|\indic_X\mathcal F_N^*\indic_Yu\|_{\ell^2_N}^2
\leq CN^{-\beta_0}\|u\|_{\ell^2_N}^2
$$
which gives~\eqref{e:fullgap-oqm} with $\beta=\beta_0/2$.
\end{proof}
Combining Propositions~\ref{l:reg-C-k-N} and~\ref{l:fullgap-oqm},
we obtain~\eqref{e:fup-oqm} for ${1\over 2}\leq \delta<1$,
finishing the proof of Theorem~\ref{t:fup-oqm}.

\medskip\noindent\textbf{Acknowledgements.}
We would like to thank Jean Bourgain for inspiring discussions
on the fractal uncertainty principle, and Maciej Zworski for many helpful
comments. We would also like to thank an anonymous referee for many
suggestions to improve the manuscript.
This research was conducted during the period SD served as
a Clay Research Fellow.


\end{document}